\documentclass[hidelinks, 12pt]{amsart}
\usepackage{tikz}
\usetikzlibrary{calc,3d}
\usepackage{xurl}
\allowdisplaybreaks
\usepackage[english]{babel}
\usepackage[pdftex,paper=a4paper,portrait=true,textwidth=450pt,textheight=675pt,tmargin=3cm,marginratio=1:1]{geometry}
\usepackage{amsfonts}
\usepackage{graphicx}
\usepackage{caption}
\usepackage{subcaption}
\usepackage{amsmath}
\usepackage{amsthm}
\usepackage{amssymb}
\usepackage{bbm}
\usepackage{cancel}
\usepackage{comment}
\usepackage{color}
\usepackage{biblatex}
\usepackage[all,cmtip]{xy}
\usepackage{graphicx}
\usepackage{mathabx}
\usepackage{tikz-cd}
\usepackage{epigraph} 
\usepackage{hyperref}
\usepackage{stmaryrd}
\usepackage{enumitem}
\usepackage{mathtools}
\usepackage[mathscr]{eucal}
\mathtoolsset{showonlyrefs}
\newtheorem{theorem}{Theorem}[section]

\newtheorem{proposition}[theorem]{Proposition}
\newtheorem{lemma}[theorem]{Lemma}
\newtheorem{claim}[theorem]{Claim}
\newtheorem{conjecture}[theorem]{Conjecture}

\theoremstyle{definition}
\newtheorem{definition}[theorem]{Definition}

\newtheorem{remark}[theorem]{Remark}

\DeclareFontFamily{OT1}{rsfs}{}
\DeclareFontShape{OT1}{rsfs}{n}{it}{<-> rsfs10}{}
\DeclareMathAlphabet{\curly}{OT1}{rsfs}{n}{it}

\newcommand\I{\mathcal I}

\newcommand\C{\mathbb C}

\newcommand\Q{\mathbb Q}

\renewcommand\k{\mathsf k}

\renewcommand\hom{\mathcal{H}{\it{om}}}

\newcommand\ext{\curly Ext}

\newcommand\INTO{\ar@{^{(}->}[r]}

\DeclareRobustCommand{\SkipTocEntry}[4]{}

\def\and{\quad\mathrm{and}\quad}

\title{Virasoro constraints for K3 surfaces and monodromy operators}
\author{Weisheng Wang}
\begin{document}

\maketitle
\newcommand{\ch} [3][2]{\operatorname{ch}^\mathrm{H}_{#2}( #3 )}
\newcommand{\plusbinomial}[3][2]{(#2 + #3)^#1}
\begin{abstract}
   The Virasoro constraints for moduli spaces of stable torsion free sheaves on a surface with only $(p,p)$-cohomology were recently proved by Bojko-Moreira-Lim. The rank 1 case, which is \emph{not} restricted to surfaces with only $(p,p)$-cohomology, was established by Moreira. We prove Virasoro constraints for K3 surfaces using Markman monodromy operators, which allow us to reduce to the rank 1 case. We also prove new Virasoro constraints in rank 0. Finally, for K3 surfaces, we introduce new Virasoro operators in negative degree which, together with the previous Virasoro operators, give a representation of Virasoro algebra with central charge $24$.
\end{abstract}

\section{Introduction}
The Virasoro operators in Gromov-Witten theory were first proposed in \cite{Eguchi_1997}, where the Virasoro operators are constructed for some Fano varieties. Here, I will recall the form of the Virasoro operators following \cite{pandharipande2003questions}. The Virasoro constraints on the moduli of stable pairs side is obtained by the GW/PT correspondence on $3$-folds \cite{moreira2020virasoro}, I will briefly review this. The Virasoro constraints on moduli of stable sheaves were first obtained on Hilbert schemes of points on a surface $S$ by restrict the stable pairs case to $S\times \mathbb{P}^1$ \cite{moreira2020virasoro}\cite{Moreira_2022}. It is then generalised to moduli of stable sheaves of higher ranks \cite{vanbree2021virasoro}\cite{bojko2022virasoro}.

\subsection{Virasoro constraint in Gromov-Witten theory}

Gromov-Witten theory is defined by integration over the moduli space of stable maps. Let $X$ be a non-singular projective variety over $\C$. A map from a connected pointed nodal curve to $X$ is a stable map if it has finite automorphism group (for more detail see chapter 24 \cite{Hori:2003ic}). A stable map $f$ represents a homology class $\beta\in H_2(X,\mathbb{Z})$ if $f_\ast[C] =\beta$. $\overline{\mathcal{M}}_{g,n}(X,\beta)$ denotes the moduli space of stable maps from $n$-pointed genus $g$ nodal curves to $X$ representing the class $\beta$, it is a proper Deligne-Mumford stack.
There are $n$ evaluation maps $\operatorname{ev}_i:\overline{\mathcal{M}}_{g,n}(X,\beta)\to X$ given by:
\[\operatorname{ev}_i(\Sigma,p_1,\dots,p_n,f)=f(p_i)\quad (1\leq i\leq n).\]
At each point $[\Sigma,p_1,\dots,p_n,f]$ of $\overline{\mathcal{M}}_{g,n}(X,\beta)$, the cotangent line to $\Sigma$ at point $p_i$ is a one-dimensional vector space; those spaces glue together to give a line bundle $
\mathbb{L}_i$ called the $i$th tautological line bundle. Define $\psi_i:=c_1(\mathbb{L}_i)$.

Let $\{\gamma_a\}$ be a homogeneous basis of $H^\ast(X,\mathbb{C})$. The descendent Gromov-Witten invariants of $X$ are:
\begin{align}
   \langle \tau_{k_1}(\gamma_{a_1})\dots\tau_{k_n}(\gamma_{a_n}) \rangle_{g,\beta}^X=\int_{\left[\overline{\mathcal{M}}_{g,n}(X,\beta)\right]^\text{vir}} \psi_1^{k_1}\operatorname{ev}_1^\ast(\gamma_{a_1})\dots\psi_n^{k_n}\operatorname{ev}_n^\ast(\gamma_{a_n}).
\end{align}
Let $\{ t^a_k \}$ be a set of variables, and define the generating function $F^X(t,\lambda)$ as:
\[F^X=\sum_{g\geq 0} \lambda^{2g-2} \sum_{\beta\in H_2(X,\mathbb{Z})} q^\beta \sum_{n\geq 0} \frac{1}{n!} \sum_{\substack{a_1\dots a_n\\k_1\dots k_n}}t^{a_n}_{k_n}\dots t^{a_1}_{k_1} \langle \tau_{k_1}(\gamma_{a_1})\dots \tau_{k_n}(\gamma_{a_n}) \rangle^X_{g,\beta}. \]
Also define the full partition function $Z^X = \exp(F^X)$. $Z^X$ is the partition function corresponds to the standard disconnected Gromov-Witten bracket $\langle,\rangle^{X,\bullet}_{g,\beta}$, where the domain nodal curve could be disconnected:
\begin{align}
    Z^X=\sum_{g\geq 0} \lambda^{2g-2} \sum_{\beta\in H_2(X,\mathbb{Z})} q^\beta \sum_{n\geq 0} \frac{1}{n!} \sum_{\substack{a_1\dots a_n\\k_1\dots k_n}}t^{a_n}_{k_n}\dots t^{a_1}_{k_1} \langle \tau_{k_1}(\gamma_{a_1})\dots \tau_{k_n}(\gamma_{a_n}) \rangle^{X, \bullet}_{g,\beta}.
\end{align}
A set of formal differential operators $\{ L_k \}_{k\geq -1}$ are defined in \cite{pandharipande2003questions}, They are defined in variables $t_k^a$ and only depend upon the intersection pairing $g_{ab}=\int_X\gamma_a\cup \gamma_b$, with $\gamma_a,\gamma_b\in H^\ast(X,\C)$, the Hodge decomposition $\gamma_a\in H^{p_a,q_a}(X,\mathbb{C})$ and the action of $c_1(X)$ on $\{\gamma_a\}$. For the precise form of operators $\{ L_k \}_{k\geq -1}$, one can refer section 4 of \cite{pandharipande2003questions}. Those operators satisfy the Virasoro bracket,
\[[L_k,L_l]=(k-l)L_{k+l}.\]
The Virasoro conjecture in Gromov-Witten theory states as follows:

\begin{conjecture}
For all non-singular projective varieties $X$, $L_k(Z^X)=0$.
\end{conjecture}
This conjecture has been proven for curves $C_g$ of genus $g$ and nonsingular projective toric varieties.
\subsection{GW/Pairs correspondence for 3-folds and Virasoro constraints for stable pairs}
The moduli space of stable maps $\overline{\mathcal{M}}_{g,n}(X,\beta)$ is essentially based upon the geometry of curves in $X$, there is another way to approach the moduli of curves in $X$, which is the moduli of stable pairs. 
\begin{definition}
    A stable pair $(F,s)$ on $X$ is a coherent sheaf $F$ on $X$ and a section $s\in H^0(X,F)$ satisfying the following two stability conditions: (1) $F$ is pure of dimension 1, (2) the section $s:\mathcal{O}_X\to F$ has cokernel of dimensional 0.
\end{definition}
Given a stable pair $\mathcal{O}_X\to F$, the kernel of $s$ defines a Cohen-Macaulay subcurve $C\subset X$, i.e. $C$ has no embedded points; the support of the cokernel of $s$ defines a $0$-dimensional subscheme of $C$. To a stable pair, the Euler characteristic and the class of the support $C$ of $F$ is associated:
\[\chi(F)=n\in\mathbb{Z}\quad\text{and} \quad [C]=\beta \in H_2(X,\mathbb{Z}).\]
For a fixed $n$ and $\beta$, there is a projective moduli space of stable pairs $P_n(X,\beta)$, it is non empty only if $\beta$ is an effective curve class. This moduli space is studied in \cite{Pandharipande_2009}, it has been shown that $P_n(X,\beta)$ is a fine moduli space with a universal stable pair $(\mathbb{F},s)$ over $X\times P_n(X,\beta)$. Let $\pi_X$ and $\pi_P$ be the projection to the first and second factor, then one can define the descendent class by:
\begin{align*}
\operatorname{ch}_k(\gamma)=\pi_{P\ast}\left( \operatorname{ch}_k(\mathbb{F}-\mathcal{O}_{X\times P_n(X,\beta)})\cdot \pi_X^\ast \gamma \right).
\end{align*}
The invariant in the stable pair theory has the following form:
\begin{align*}
    \langle \operatorname{ch}_{k_1}(\gamma_1) \dots \operatorname{ch}_{k_m}(\gamma_m) \rangle^{X,\text{PT}}_{n,\beta}:=\int_{\left[ P_n(X,\beta) \right]^\text{vir}}\prod_{i=1}^m \operatorname{ch}_{k_i}(\gamma_i).
\end{align*}
The moduli of stable maps and stable pair are both based upon the geometry of curves in $X$, therefore one may hope there exists some links between those two descendent theories. Indeed, for non-singular projective 3-folds, this correspondence conjecturally holds: the Gromov-Witten and stable pairs descendent series are related after a change of variables.

Let $\mathbb{D}_{\text{PT}}^X$ be the commutative $\mathbb{Q}$-algebra with generators \[\{\operatorname{ch}_i(\gamma)|i\geq 0,\gamma\in H^\ast (X,\Q)\}\] subject to the nature relations
\begin{align*}
\operatorname{ch}_i(\lambda\cdot\gamma)&=\lambda\operatorname{ch}_i(\gamma)\\
\operatorname{ch}_i(\hat{\gamma}+\gamma)&=\operatorname{ch}_i(\hat{\gamma})+\operatorname{ch}_i(\gamma),
\end{align*}
For $\lambda\in \mathbb{Q}$ and $\gamma,\hat{\gamma}\in H^\ast (X)$. Define $\mathbb{D}^X_{\text{GW}}$ similarly as $\mathbb{D}_{\text{PT}}^X$ using generators $\tau_{i}(\gamma)$'s. The GW/PT correspondence is a linear map $\mathfrak{C}^\bullet:\mathbb{D}_{\text{PT}}^{X\bigstar}\to\mathbb{D}^X_{\text{GW}}$, where $\mathbb{D}_{\text{PT}}^{X\bigstar}$ is a subalgebra of $\mathbb{D}_{\text{PT}}^{X}$ called \textit{essential descendants} and the map $\mathfrak{C}^\bullet$ is defined on monomials, for the precise definition of $\mathfrak{C}^\bullet$ and $\mathbb{D}_{\text{PT}}^{X\bigstar}$ one can refer \cite{moreira2020virasoro}. In toric 3-fold case, the precise correspondence statement is:
\begin{theorem}[Theorem 6 of \cite{moreira2020virasoro}]
    Let $X$ be a nonsingular projective toric 3-fold. Let $D\in \mathbb{D}_{\mathrm{PT}}^{X\bigstar}$, $\beta \in H_2(X,\mathbb{Z})$ with $d_\beta=\int_{\beta}c_1(X)$. Then the $\mathrm{GW/PT}$ correspondence holds:
    \[(-q)^{-d_\beta / 2}\left(\sum_{n\in\mathbb{Z}}q^n\left\langle D\right\rangle_{n,\beta}^{X, \mathrm{PT}}\right)=(-i u)^{d_\beta}\left(\sum_{g\in\mathbb{Z}}u^{2g-2}\left\langle\mathfrak{C}^{\bullet}\left(D\right)\right\rangle_{g,\beta}^{X, \mathrm{GW}}\right)\]
    after the change of variable $-q=e^{iu}$. 
\end{theorem}

In \cite{moreira2020virasoro}, Virasoro constraints on moduli of stable pairs are proven using the Virasoro constraints for the Gromov-Witten theory of toric 3-folds and the above GW/PT correspondence. The Virasoro operators $\mathcal{L}^{\text{PT}}_k$ with $k\geq -1$ are defined as a operators on $\mathbb{D}_{\text{PT}}^{X}$. For the precise definition one can refer \cite{moreira2020virasoro}. The Virasoro constraints states as follows:
\begin{conjecture}[\cite{moreira2020virasoro}]
    Let $X$ be a nonsingular projective 3-fold with only $(p, p)$-cohomology, and let $\beta \in H_2(X, \mathbb{Z})$. For all $k \geqslant-1$ and $D \in \mathbb{D}_{\mathrm{PT}}^X$, we have
$$
\left\langle\mathcal{L}_k^{\mathrm{PT}}(D)\right\rangle_\beta^{X, \mathrm{PT}}=0,
$$
where 
$$
\left\langle D\right\rangle_\beta^{X, \mathrm{PT}}:=\sum_{n \in \mathbb{Z}} q^n \int_{\left[P_n(X, \beta)\right]^{v i r}} D.
$$
\end{conjecture}
The statement of this conjecture about stationary descendants for non-singular projective toric 3-folds is proven in \cite{moreira2020virasoro}. A special case is when $X=S\times \mathbb{P}^1$ where $S$ is a smooth projective toric surface, then the Virasoro constraints for this toric 3-fold $X$ are:
\[\forall k \geqslant-1, \quad\left\langle\mathcal{L}_k^{\mathrm{PT}} \prod_{i=1}^r \mathrm{ch}_{m_i}\left(\gamma_i \times \boldsymbol{\operatorname{p}}\right)\right\rangle_{n\left[\mathbb{P}^1\right]}^{X, \mathrm{PT}}=0,\]
where $\gamma_i\in H^\ast(S)$, $\boldsymbol{\operatorname{p}}\in H^2(\mathbb{P}^1)$ is the point class and $[\mathbb{P}^1]\in H_2(X)$ is the fiber class. Specializing to the space $P_n(S\times \mathbb{P}^1,n[\mathbb{P}^1])\cong \operatorname{Hilb}^n(S) $, one gets a new set of Virasoro constraints for tautological classes on $\operatorname{Hilb}^n(S)$ for toric surfaces. In \cite{Moreira_2022}, this constraint is proven for simply connected surfaces, I will recall the full form of the Virasoro operators in this Hilbert scheme of points setting in the next section. 

The Hilbert scheme of points can be viewed as a moduli of stable sheaves of rank $1$, therefore one may expect that there is a Virasoro constraint for the moduli of stable sheaves. In \cite{vanbree2021virasoro}, D. van Bree has proposed a Virasoro conjecture for the moduli space of Gieseker semistable sheaves for surfaces with only $(p,p)$-cohomology. In \cite{bojko2022virasoro}, Bojko, Lim and Moreira have proved the Virasoro constraint for moduli spaces of stable torsion-free sheaves on any curve and on surfaces with only $(p,p)$-cohomology classes, in particular the conjecture of D. van Bree is proven.

The following conjecture generalizes Moreira’s Virasoro constraints on Hilbert schemes of points on simply connected surfaces\cite{Moreira_2022} and van Bree’s Virasoro constraints on moduli of sheaves on surfaces with only $(p,p)$ cohomology\cite{vanbree2021virasoro}. It is (implicitly) conjectured in Section 2.8 of \cite{bojko2022virasoro}. I will provide a proof for this conjecture in the case of K3 surface. 

\begin{conjecture}
    Let $S$ be a smooth projective simply connected surface over $\C$ and let $H$ be a fixed polarisation. Choose numbers $r>0$ and $c_2$ and a line bundle $L$. Let $M=M_S^H(r,L,c_2)$ be the moduli space of Gieseker semistable sheaves of rank $r$, with determinant $L$ and second Chern class $c_2$. Assume that all semistable sheaves are also stable and $M$ has a (twisted) universal sheaf $\mathcal{F}$. Let $\mathbb{D}^S$ be the \textit{holomorphic descendents} defined in the section \ref{sectionVirasoro}. Let $\xi :\mathbb{D}^S\to H^\ast(M,\C)$ be the geometric realization defined in \eqref{geomreal} and let $\mathcal{L}_k, k\geq -1$ be the Virasoro operators defined in \eqref{kpositive}. Then $\forall k\geq -1$ and $\forall D\in \mathbb{D}^S$, we have
    $$\int_{[M]^\mathrm{vir}}  \xi_{\mathcal{F}\otimes(\operatorname{det}\mathcal{F})^{-1/r}}\left(\mathcal{L}_k D \right)=0.\footnote{The geometric realization using $\mathcal{F}\otimes(\operatorname{det}\mathcal{F})^{-1/r}$ is equivalent to the geometric realization using $\boldsymbol{\operatorname{p}}$-normalized sheaf in \cite{bojko2022virasoro} by Remark 2.17 in the same paper.}$$
\end{conjecture}

In Theorem \ref{mainthm}, the above conjecture is proven for $S$ being a K3 surface. The idea of the proof is to use the Markman operator defined in \cite{markman2005monodromy} which relates cohomology rings of two moduli spaces of stable sheaves on K3 surfaces, more precisely: given an isometry of Mukai lattice $g:H^\ast(S_1,\C)\to H^\ast(S_2,\C)$ of two K3 surfaces $S_1,S_2$, Markman defines an operator $\gamma(g):H^\ast(M_1(v),\C)\to H^\ast(M_2(g(v)),\C)$, where $M_i(v)$ is the moduli of stable sheaves on $S_i$ with Mukai vector $v$ with corresponding polarizations. It turns out that the Markman operator $\gamma(g)$ is an isometry with respect to Poincaré pairing which allows us to equal integrals on both moduli spaces. In \cite{Oberdieck_2022}, Oberdieck proves an universality result about descendent integrals over moduli space of stable sheaves on K3 surfaces using the Markman operator. In this paper, I will use this universality result to transform the descendent integrals on moduli of stable sheaves to descendent integrals on Hilbert schemes of points where the Virasoro conjecture is proven. 

In section \ref{rankzero}, the rank zero case is considered and a modified version of the Virasoro constraints will be proposed and proven in proposition \ref{rankzeroprop}. In this rank zero case, a new $S_k$ operator will be constructed and one need to use $Y$-normalised universal sheaf (in the sens of \cite{bojko2022virasoro}) for the geometric realization for a $Y\in H^{1,1}(S)$.

I will also propose a set of negative Virasoro operators on K3 surface. Those negative Virasoro operators combined with the existing $L_{k\geq -1}$ operators will satisfy the Virasoro algebra with central charge $e(S)$. This is compatible with the result in \cite{bojko2022virasoro} for $p_g=0$ surfaces. In the Gromov-Witten side, the Virasoro operators in \cite{Hori:2003ic} also give the same conformal weight. 

In section 2, I will recall the Virasoro constraints of Moreira and van Bree and formulate the Virasoro constraints on K3 surfaces. In section 3, I will recall the Markman operator following G.Oberdieck's expositions \cite{Oberdieck_2022}. In section 4, I will prove the Virasoro constraint on moduli of positive rank stable sheaves on K3 surfaces. In section 5, I will consider the rank 0 case. In section 6, I will show that the set of negative Virasoro operators on K3 surface that I proposed satisfy the Virasoro algebra with central charge $e(S)$.
\subsection{Acknowledgement}
The author would like to thank his supervisor Martijn Kool for many useful suggestions. He is also grateful to Woonam Lim for useful conversations. The author is supported by the ERC Consolidator Grant FourSurf 101087365.

\section{Virasoro constraints}
\label{sectionVirasoro}
Let $S$ be a non-singular and projective K3 surface. There is a bilinear form, called $\textit{Mukai pairing}$, on $\Lambda:= H^\ast(S,\mathbb{Z})$ defined as:
\begin{align*}
    (x,y):=-\int_S x^\vee y, 
\end{align*}
where $\bullet^\vee$ is defined as follows: if one writes $x = (r,D,n)$ as the decomposition of degree then $x^\vee = (r,-D,n)$. This pairing is symmetric, unimodular, of signature $(4,20)$ and the resulting lattice is called the $\textit{Mukai lattice}$. For $x=(r,D,n)\in \Lambda$, I will also write
\begin{align}
\label{notation}
    \operatorname{rk}(x)=r,\quad c_1(x)=D,\quad v_2(x)=n.
\end{align}

For a coherent sheaf $\mathcal{F}$ on $S$, its Mukai vector is defined to be $v=\operatorname{ch}(\mathcal{F})\cdot \sqrt{\operatorname{td}_S}$.
Let $v\in\Lambda$ and an ample line bundle $H$ be chosen such that $M:=M_H(v)$, the moduli space of $H$-semistable sheaves on $S$ with Mukai vector $v$, does not contain any strictly semistable sheaves and is a smooth projective and admits a (twisted) universal sheaf.
For example, one could choose $v\in H^{1,1}(S)$, $v\neq0$ and $(v,v)\geq -2$ and $v$ is not a multiple of a class by an integer larger than $1$, those choices of $v$ are called \textit{effective} and \textit{primitive} in the sense of \cite{markman2005monodromy}. Also, for such $v$, there always exists an ample line bundle $H$ on $S$ such that $M_H(v)$ has the above mentioned properties \cite{markman2005monodromy}. 

Let $\mathbb{D}^S$ be the commutative algebra generated by symbols called \textit{holomorphic descendents} of the form:
\[\ch{i}{\gamma}\quad\text{ for }i\geq 0, \gamma\in H^\ast(S,\C)\] subject to the linearity relations\[\ch{i}{\lambda_1\gamma_1+\lambda_2\gamma_2}=\lambda_1\ch{i}{\gamma_1}+\lambda_2\ch{i}{\gamma_2}\] for $\lambda_1,\lambda_2\in \mathbb{C}$. I also write $\ch{\bullet}{\gamma}$ for the element 
\[\ch{\bullet}{\gamma} =\sum_{i\geq 0} \ch{i}{\gamma} \in \mathbb{D}^S.\footnote{The elements in $\mathbb{D}^S$ can only allow finite sums of descendants, here we may view this infinite sum lives in a completion of $\mathbb{D}^S$. }\]

Consider the moduli space $M:=M_H(v)$ with $r:=\operatorname{rk}(v)\geq 1$ and the product $M\times S$, let $\pi_M$ and $\pi_S$ be the projection from $M\times S$ to $M$ and $S$. Let $\mathcal{F}$ be a universal sheaf or a twisted universal sheaf. Let me recall the definition in \cite{bojko2022virasoro} of the geometric realization with respect to $\mathcal{F}$ on $M\times S$ as the algebra homomorphism 
\begin{align}
\label{geomreal}
    \xi_{\mathcal{F}}:\mathbb{D}^S \to H^\ast(M),
\end{align}
which acts on generators $\ch{i}{\gamma}$ with $\gamma\in H^{p,q}(S)$ as 
\[\xi_{\mathcal{F}}\left( \ch{i}{\gamma} \right)=\pi_{M\ast}\left( \operatorname{ch}_{i+\operatorname{dim}(S)-p}(\mathcal{F}) \pi_S^\ast \gamma \right)\]

Next, I will define the Virasoro operators for K3 surfaces. I combined the form of Virasoro operators in \cite{Moreira_2022} and in \cite{vanbree2021virasoro}. For $k\geq -1$, I define operators $R_k$, $T_k$, $S_k$ on $\mathbb{D}^S$ as follows:

$\bullet$ The operator $R_k:\mathbb{D}^S\to\mathbb{D}^S$ is defined as a derivation by fixing its action on the generators: given $\gamma\in H^{p,q}(S)$, 
\begin{align*}
    R_k (\operatorname{ch}_i^\mathrm{H}(\gamma))&=\left(\prod_{j=0}^k(i+j)\right)\operatorname{ch}_{i+k}^\mathrm{H}(\gamma)
\end{align*}
I take the following conventions: the above product is $1$ if $k=-1$ and $\ch{i+k}{\gamma}=0$ if $i+k<0$.

$\bullet$ The operator $T_k:\mathbb{D}^S\to\mathbb{D}^S$ is the operator of multiplication by a fixed element of $\mathbb{D}^S$:
\begin{align*}
    T_k&=\sum_{i+j=k}(-1)^{\operatorname{dim}(S)-p^L}i!j!\operatorname{ch}_{i}^\mathrm{H}\operatorname{ch}_{j}^\mathrm{H}(\text{td}_S)
\end{align*}
where 
$(-1)^{\operatorname{dim}(S)-p^L}i!j!\operatorname{ch}_{i}^\mathrm{H}\operatorname{ch}_{j}^\mathrm{H}(\text{td}_S)$ \footnote{In \cite{bojko2022virasoro}, $T_k$ operators are also defined for curves, in this case $\operatorname{dim}(S)$ should be replaced by the dimension of the curve.} is defined as follows: let $\Delta:S\to S\times S$ be the diagonal map and let 
\[\sum_{t}\gamma_t^L\otimes \gamma_t^R = \Delta_\ast\operatorname{td}_S\]
be the K\"unneth decomposition of $\Delta_\ast\operatorname{td}_S$ such that $\gamma^L_t\in H^{p^L_t,q^L_t}(S)$ for some $p^L_t,q^L_t$, Then 
\[(-1)^{\operatorname{dim}(S)-p^L}i!j!\operatorname{ch}_{i}^\mathrm{H}\operatorname{ch}_{j}^\mathrm{H}(\text{td}_S)=\sum_t(-1)^{\operatorname{dim}(S)-p^L_t}i!j!\operatorname{ch}_{i}^\mathrm{H}(\gamma_t^L)\operatorname{ch}_{j}^\mathrm{H}(\gamma_t^R).\]

$\bullet$ For $\alpha\in H^\ast(S)$, define the derivation $R_{-1}[\alpha]$ by $R_{-1}[\alpha](\ch{i}{\gamma})=\ch{i-1}{\alpha\gamma}$. For $k\geq -1$, $S_k:\mathbb{D}^S\to\mathbb{D}^S$ is defined 
\begin{align}
    S_k (D) = -\frac{(k+1)!}{r}\sum_{\{i | p^L_i=0\}}R_{-1}[\delta^L_i]\left(\operatorname{ch}_{k+1}^\mathrm{H}(\delta^R_i)D\right),\quad\forall D\in\mathbb{D}^S,
\end{align}
where $r$ is the rank of the sheaves that the moduli $M$ parameterize and the sum runs over the terms $\delta_i^L\otimes \delta_i^R$ of the K\"unneth decomposition of $\Delta_\ast \mathbbm{1}\in H^\ast(S\times S)$ such that $p_i^L=0$.

Finally define operators $L_k$ and $\mathcal{L}_k$ for $k\geq-1$ as:
\begin{align}
\begin{split}
\label{kpositive}
    L_k&=R_k+T_k\\
    \mathcal{L}_k&=R_k+T_k+S_k
    \end{split}
\end{align}

\begin{remark}
    Denote by $S^{[n]}$ the Hilbert scheme of points on $S$ parameterizing $0$ dimensional subschemes of $S$ with length $n$. Let $\mathcal{I_\mathcal{Z}}$ be the ideal sheaf of the universal subscheme; equivalently one can write $\mathcal{I_\mathcal{Z}}=\mathcal{O}_{S^{[N]}\times S}-\mathcal{O}_{{\mathcal{Z}}}$. When viewing $S^{[n]}$ as a moduli of rank one sheaves, the universal sheaf $\mathcal{F}$ is $\mathcal{I_\mathcal{Z}}$, in this case one has $\operatorname{ch}(\mathcal{F}\otimes(\operatorname{det}\mathcal{F})^{-1})=\operatorname{ch}(\mathcal{F})$. Therefore, in the case of $S^{[n]}$, the above defined Virasoro operators become the Virasoro operators of Moreira in \cite{Moreira_2022}: in \cite{Moreira_2022}, the geometric realization $\xi_{I_\mathcal{Z}}$ is used. If one considers a surface $S$ with only $(p,p)$ cohomology and a moduli space $M$ of Gieseker semistable sheaves of of rank $r\geq 1$ as the case in \cite{vanbree2021virasoro}, then the above defined Virasoro operators become the Virasoro operators of van Bree \cite{vanbree2021virasoro}.
\end{remark}

One of the main result of this paper is:
\begin{theorem}
\label{mainthm}
    Let $S$ be a non-singular and projective K3 surface. Let $D\in\mathbb{D}^S$. Let $M_H(v)$ be a moduli of sheaves on $S$ as above with $r:=\operatorname{rk}(v)\geq 1$, and let $\mathcal{F}$ be a universal sheaf. Define $\overline{\mathcal{F}}:=\mathcal{F}\otimes(\operatorname{det}\mathcal{F})^{-1/r}$ as an element of the rational K-theory of $M_H(v)\times S$\footnote{ $(\operatorname{det}\mathcal{F})^{-1/r}$ might not exist as a line bundle, also the geometric realization is well define is this case, since it only involves Chern characters} ($\overline{\mathcal{F}}$ is independent of the choice of the universal sheaf $\mathcal{F}$ as in \cite{vanbree2021virasoro}).  Then for $k\geq -1$, we have:
    $$\int_{M_H(v)}  \xi_{\overline{\mathcal{F}}}\left(\mathcal{L}_k D \right)=0.$$
\end{theorem}

\section{Markman's monodromy operator}
Let $S$ be a non-singular projective K3 surface and let $M:=M_H(v)$ be a moduli of sheaves on $S$ as the section \ref{sectionVirasoro}, but in this section, it is possible to take $\operatorname{rk}(v)=0$. Morphisms $\pi_M$ and $\pi_S$ are defined in the same way as the section \ref{sectionVirasoro}.

Define the morphism $\theta_\mathcal{F}:\Lambda\to H^2(M,\mathbb{Q})$ as follows: 
\[\theta_\mathcal{F}(x)=[\pi_{M\ast}(\operatorname{ch}(\mathcal{F})\pi_S^\ast(\sqrt{\text{td}_S}\cdot x^\vee))]_2,\]
where $[ \bullet ]_k$ means take the real degree $k$ component of a cohomology class.

The morphism $B:H^\ast(S,\mathbb{Q})\to H^\ast(M,\mathbb{Q})$ is defined by
\begin{align}
\label{Bmorphism}
    B(x)=\pi_{M\ast}(u_v\cdot x^\vee),
\end{align}
 where \[u_v:=\exp\left( \frac{\theta_{\mathcal{F}}(v)}{(v,v)}\right)\cdot 
\operatorname{ch}(\mathcal{F})\cdot\sqrt{\text{td}_S} \in H^\ast(M\times S,\mathbb{Q}),\]
where some pull-backs $\pi_M^*$, $\pi_S^*$ are suppressed. One can check that $B$ is independent of the choice of the (twisted) universal sheaf $\mathcal{F}$.

Let $S_1$ and $S_2$ be two non-singular projective K3 surfaces with polarizations $H_1,H_2$. Let $g:H^\ast(S_1,\mathbb{Z})\to H^\ast(S_2,\mathbb{Z})$ be an isometry of Mukai lattices, assume $v_1$, $v_2$ are two vectors in the Mukai lattice satisfying assumptions of the section \ref{sectionVirasoro} and $v_2=g(v_1)$. Let $M_i:=M_{H_i}(v_i)$ for $i=1,2$.
Markman defined the transformation $\gamma(g): H^\ast(M_1,\mathbb{C}) \to H^\ast(M_2,\mathbb{C})$ in \cite{markman2005monodromy}:
\begin{align}
    \gamma(g)(x) = \pi_{2\ast}\left(c_{\operatorname{dim}(M)}\big[-\pi_{13\ast}\left(\pi_{12}^\ast \left( (1\otimes g) u_{v_1}\right)^\vee \cdot \pi_{23}^\ast u_{v_2} \right) \big]\cdot \pi_{1}^\ast x\right),
\end{align}
where $\pi_{ij}$ ($\pi_k$) is the projection of $M_1\times S_2 \times M_2$ to the $(i,j)$-th ($k$-th) factor. 
The main properties of $\gamma(g)$ are given in the following theorem: 
\begin{theorem}[Markman]
\label{markmantheorem}
Let $S_1,S_2$ and $v_1,v_2$ as above. For any isometry $g:H^\ast(S_1,\mathbb{C})\to H^\ast(S_2,\mathbb{C})$ such that $g(v_1)=v_2$, $\gamma(g)$ is the unique operator such that: 
\begin{enumerate}
\item[$\mathrm{(i)}$] $\gamma(g)$ is a degree-preserving ring isomorphism and is a isometry with respect to the Poincaré pairing: $\langle x,y\rangle = \int_M xy$ for all $x,y \in H^\ast(M,\mathbb{Q})$. 
\item[$\mathrm{(ii)}$] $(\gamma(g)\otimes g)(u_{v_1})=u_{v_2}$.
\item[$\mathrm{(iii)}$] $\gamma(g_1)\circ \gamma(g_2) = \gamma(g_1g_2)$ and $\gamma(g)^{-1} = \gamma(g^{-1})$(if it makes sens), where $g_1,g_2$ are two isometries.
\item[$\mathrm{(iv)}$] $\gamma(g)\left(c_k(T_{M_1})\right)=c_k(T_{M_2})$.
\end{enumerate}
\end{theorem}
The above properties imply that 
\begin{equation} \label{eqn:key}
\int_{M_1} \sigma = \int_{M_2}\gamma(g)(\sigma) \quad\forall \sigma\in H^\ast(M_1,\mathbb{Q}).
\end{equation}

In fact the property $(\gamma(g)\otimes g)(u_{v_1})=u_{v_2}$ can be expressed in a nicer way using the $B$ morphism defined in \eqref{Bmorphism}: 
\begin{lemma}[Lemma 2.7 of \cite{Oberdieck_2022}]
\label{lemmamarkman}
    Let $M_i,v_i$ with $i=1,2$ and $g$ as above.
    Let the canonical morphism $B:H^\ast(S,\mathbb{Q}) \to H^\ast(M_i,\mathbb{Q})$ be defined as \eqref{Bmorphism} and write $B_k(x)$ for the component in degree $2k$. Let $f:H^\ast(M_1,\mathbb{Q})\to H^\ast(M_2,\mathbb{Q})$ be a degree preserving isometric ring isomorphism. Then the following are equivalent:
    \begin{enumerate}
    \item[$\mathrm{(i)}$] $(f\otimes g)(u_{v_1})=u_{v_2}$,
    \item[$\mathrm{(ii)}$] $f(B(\gamma))=B(g\gamma)$ for all $\gamma\in H^\ast (S_1,\mathbb{Q})$.
    \end{enumerate}
\end{lemma}
Therefore, we have the following commutative diagram.
\[
    \begin{tikzcd}
H^\ast(S_1,\Q) \arrow{r}{g} \arrow[swap]{d}{B} & H^\ast(S_2,\Q) \arrow{d}{B} \\%
H^\ast(M_1,\Q) \arrow{r}{\gamma(g)}& H^\ast(M_2,\Q)
\end{tikzcd}\]

Now, let $x_1\dots x_k \in \Lambda $, consider the following descendent integrals over $M$:
\[\int_M P(B_i(x_j),c_l(T_M)),\]
where $P(t_{ij},u_l)$ is a polynomial with $\mathbb{Q}$-coefficients depending on the variables $t_{ij}$, with $j=1,\dots,k$, $i\geq 0$ and $u_l$, $l\geq 1$.

Using Markman's operator, Oberdieck proved a universality property for integrals over $M$; this roughly means that integral $\int_M P(B_i(x_j),c_l(T_M))$ only depends on the polynomial $P$, the dimension of $M$ and pairings $(v,x_i)$, $(x_i,x_j)$ for all $i,j$ i.e. the intersection matrix
\begin{align*}
    \begin{pmatrix}
(v,v) & (v,x_i)_{i=1}^k \\
(x_i,v)_{i=1}^k & (x_i,x_j)_{i,j=1}^k
\end{pmatrix}.
\end{align*}

I will now explain this in detail. In fact, given the data $(M(v),x_i)$ as above, the Lemma 2.11 of \cite{Oberdieck_2022} shows that there exists $y_i\in \Lambda_{\mathbb{C}}$ which have the same above intersection matrix, and satisfy
\[\int_MP(B_i(x_j),c_l(T_M))=\int_MP(B_i(y_j),c_l(T_M))\]
and $\operatorname{Span}(v,y_1,\dots,y_k)$ is a non-degenerate subspace of $\Lambda_{\mathbb{C}}$ (i.e. the restriction of the inner product of $\Lambda_{\mathbb{C}}$, induced by the Mukai pairing, onto the subspace is non-degenerate).
The condition $\operatorname{dim} M>2$ is used to obtain the Lemma 2.11 of \cite{Oberdieck_2022}, therefore if one always assume $\operatorname{dim} M>2$ then given two pairs $(M(v),x_i)$ and $(M(v'),x_i')$ with the same intersection matrix, one may always assume that $\operatorname{Span}(v,x_1,\dots,x_k)$ and $\operatorname{Span}(v',x_1',\dots,x_k')$ are non-degenerate. 

Then, one can use the following lemma:
\begin{lemma}[Lemma 2.13 of \cite{Oberdieck_2022} ]
\label{Oberdlemma}
    Let $V$ be a finite-dimensional $\mathbb{C}$-vector space with a $\mathbb{C}$-linear inner product. Let $v_1,\dots,v_k\in V$ and $w_1,\dots,w_k\in V$ be lists of vectors such that
    \begin{enumerate}
    \item[$\mathrm{(i)}$] $\operatorname{Span}(v_1,\dots,v_k)$ is non-degenerate,
    \item[$\mathrm{(ii)}$] $\operatorname{Span}(w_1,\dots,w_k)$ is non-degenerate, 
    \item[$\mathrm{(iii)}$] $\langle v_i,v_j \rangle = \langle w_i,w_j \rangle$ for all $i,j$.
    \end{enumerate}
    Then there exists an isometry $\phi:V\to V$ such that $\phi(v_i)=w_i$ for all $i$.
 \end{lemma}
Since I always assume that $\operatorname{Span}(v,x_1,\dots,x_k)$ and $\operatorname{Span}(v',x_1',\dots,x_k')$ are non-degenerate, by above lemma, there exists an isometry $g:H^\ast(S,\mathbb{C})\to H^\ast(S',\mathbb{C})$ taking $(v,x_1,\dots,x_k)$ to $(v'=gv,x_1'=gx_1,\dots,x_k'=gx_k)$. 
 Therefore by the properties of the Markman operator and morphism $B$ one has:
 \begin{align*}
     \int_{M(v)} P(B_i(x_j),c_l(T_{M(v)})) &=  \int_{M(v')} \gamma(g)P(B_i(x_j),c_l(T_{M(v')}))\\
     &=  \int_{M(v')} P(B_i(gx_j),c_l(T_{M(v')}))\\
     &=  \int_{M(v')} P(B_i(x_j'),c_l(T_{M(v')})),
 \end{align*}
 where the first equality is by the Theorem \ref{markmantheorem} and the second equality is by the Lemma \ref{lemmamarkman}.
This leads to the following theorem:
\begin{theorem}[Oberdieck \cite{Oberdieck_2022}]
\label{theoremUnivers}
let $P(t_{ij},u_r)$ be a polynomial depending on the variables
$t_{ij}$, $j = 1,...,k,$   $i \geq 0$, and $u_l$, $l\geq 1$. Let also $A = (a_{ij})^k_{i,j=0}$ be a $(k + 1) \times (k + 1)$-matrix. Then there exists $I(P,A)\in \mathbb{Q}$ ($I(P,A)$ is a rational number only depending on $P,A$) such that for any $M=M_H(v)$ with $dim(M)>2$ and for $x_1,...,x_k\in \Lambda$ also with $v\in\Lambda$ as above.
\[
\begin{pmatrix}
(v,v) & (v,x_i)_{i=1}^{k} \\
(x_i,v)_{i=1}^k & (x_i,x_j)_{i,j=1}^k
\end{pmatrix}=A,\]
such that
\[\int_M P(B_i(x_j),c_l(T_M)) = I(P,A).\]
\end{theorem}

\section{Positive rank Virasoro constraints}
\label{positiverank}
\subsection{$\operatorname{dim}M>2$ case }
\label{bigger2}
As in Section \ref{sectionVirasoro}, let $S$ be a non-singular projective K3 surface, let $v=(r,D,n)\in \Lambda$ be a Mukai vector with $r=\operatorname{rk}(v)>0$. Assume $\operatorname{dim}(M_H(v))>2$.
I would like to use the universality result the Theorem \ref{theoremUnivers} introduced in the previous section to prove the Virasoro constraints for positive rank on $M_H(v)$.
Recall that, for all $\gamma\in H^\ast(S,\C)$ the symbol $\operatorname{ch}^\mathrm{H}_i(\gamma)$ is mapped to $H^\ast(M)$ by:
\[\operatorname{ch}^\mathrm{H}_\bullet(\gamma) = \pi_{M\ast}\left(\operatorname{ch}\left(\mathcal{F}\otimes \det{(\mathcal{F})}^{-1/\text{rk}(v)}\right)\pi_S^\ast (\gamma)\right),\forall \gamma\in H^\ast(S,\C).\]

We now use following brief calculation to express this in terms of $B(\bullet)$.
\begin{align}
\begin{split}
    \label{calculation4.1}
    & \pi_{M\ast}(\operatorname{ch}(\mathcal{F}\otimes \det{(\mathcal{F})}^{-1/\text{rk}(v)})\pi_S^\ast (\gamma)) 
    \\&= \pi_{M\ast}(\operatorname{ch}(\mathcal{F})\exp{\left(-\frac{c_1(\mathcal{F})}{\text{rk}(v)}\right) \pi_S^\ast(\gamma)})\\
    &=\pi_{M\ast}\left(\operatorname{ch}(\mathcal{F})\exp{\left(-\frac{\theta_\mathcal{F}(\boldsymbol{\operatorname{p}})+c_1(v)}{\text{rk}(v)} \right)} \pi_S^\ast(\gamma)\right)\\
    &=\exp{\left(\frac{\theta_\mathcal{F}(\boldsymbol{\operatorname{p}})}{(\boldsymbol{\operatorname{p}},v)}-\frac{\theta_\mathcal{F}(v)}{(v,v)}\right)}\exp{\left(\frac{\theta_\mathcal{F}(v)}{(v,v)}\right)}\pi_{M\ast}\left(\operatorname{ch}(\mathcal{F})\pi_S^\ast\left(\exp{\left(\frac{c_1(v)}{\text{rk}(v)} \right)} \gamma^\vee \sqrt{\text{td}_S}^{-1}\right)^\vee \sqrt{\text{td}_S}\right)\\
    &=\exp{\left( B_1\left(\frac{\boldsymbol{\operatorname{p}}}{(\boldsymbol{\operatorname{p}},v)}-\frac{v}{(v,v)}\right) \right) B\left(\left(\exp{\left(\frac{c_1(v)}{\text{rk}(v)} \right)} \gamma^\vee \sqrt{\text{td}_S}^{-1}\right)\right)},
    \end{split}
\end{align}
where the notation \eqref{notation} is used, also notice that $\operatorname{dim}(M)>2$ implies $(v,v)>0$. The second equality used the fact that, by K\"unneth decomposition, $c_1(\mathcal{F}) = \pi_M^\ast(\theta_{\mathcal{F}}(\boldsymbol{\operatorname{p}}))+ \pi_S^\ast(c_1(v)) $, the last line used the fact that $B_1(x)=\theta_\mathcal{F}(x)$ for $x\in v^\perp$ (see \cite{Oberdieck_2022}). 

Therefore $\operatorname{ch}^\mathrm{H}_i(\gamma)$ for $\gamma\in H^{p,q}(S)$ is the degree $2i-p+q$ component of 
\begin{align}
\label{Bform}
    \exp{\left( B_1\left(\frac{\boldsymbol{\operatorname{p}}}{(\boldsymbol{\operatorname{p}},v)}-\frac{v}{(v,v)}\right) \right) B\left(\left(\exp{\left(\frac{c_1(v)}{\text{rk}(v)} \right)} \gamma^\vee \sqrt{\text{td}_S}^{-1}\right)\right)}.
\end{align}

Consider the evaluation of the integral over $M(v)$ of a polynomial of $\operatorname{ch}_i^\mathrm{H}(\gamma)$'s. By Theorem \ref{theoremUnivers}, one only need to look at the intersection matrix of the classes appearing in the arguments of $B$ and $B_1$ in \eqref{Bform}.

Therefore, for $L\in H^2(S,\mathbb{Q})$, one needs to keep track of \[L \exp{\left(\frac{c_1(v)}{\text{rk}(v)} \right)} ,\quad \frac{\boldsymbol{\operatorname{p}}}{\text{rk}(v)},\quad v. \]

For $\boldsymbol{\operatorname{p}}$, one needs to keep track of \[\boldsymbol{\operatorname{p}} ,\quad \frac{\boldsymbol{\operatorname{p}}}{\text{rk}(v)},\quad v. \]

For $\mathbbm{1}$, one needs to keep track of \[\exp{\left(\frac{c_1(v)}{\text{rk}(v)} \right)} \sqrt{\text{td}_S}^{-1} ,\quad \frac{\boldsymbol{\operatorname{p}}}{\text{rk}(v)},\quad v. \]

Recall that for K3 surface, I have $\sqrt{\operatorname{td}_S}=1+\boldsymbol{\operatorname{p}}$. Now, also consider the descendent of the form $\ch{k}{\mathbbm{1}+\boldsymbol{\operatorname{p}}}$.
$\operatorname{ch}^\mathrm{H}_k(\mathbbm{1}+\boldsymbol{\boldsymbol{\operatorname{p}}})$ is the degree $2k$ component of the following class 
\[\left[ \pi_{M\ast}\left(\operatorname{ch}(\mathcal{F}\otimes \det{(\mathcal{F})}^{-1/\text{rk}(v)})\pi_S^\ast (\mathbbm{1}+\boldsymbol{\boldsymbol{\operatorname{p}}})\right)\right]_k,\] which equals $\operatorname{ch}^\mathrm{H}_k(\mathbbm{1})+\operatorname{ch}^\mathrm{H}_k(\boldsymbol{\boldsymbol{\operatorname{p}}})$.
Using the fact $(\mathbbm{1}+\boldsymbol{\operatorname{p}})\sqrt{\mathrm{td_S}}^{-1}=\mathbbm{1}$; for $\operatorname{ch}_k(\mathbbm{1}+\boldsymbol{\boldsymbol{\operatorname{p}}})$ one needs to track
\[\exp{\left(\frac{c_1(v)}{\text{rk}(v)} \right)},\quad \frac{\boldsymbol{\operatorname{p}}}{\text{rk}(v)},\quad v. \]

One can choose the generators of $\mathbb{D}^S$ as 
\begin{align}
\label{basis}
    \{\operatorname{ch}_j^\mathrm{H}(L_i),\operatorname{ch}^\mathrm{H}_k(\boldsymbol{\operatorname{p}}),\operatorname{ch}_l^\mathrm{H}(\mathbbm{1}+\boldsymbol{\operatorname{p}})|j,k,l\in\mathbb{N}, \{L_i\}_i \text{ forms a basis of } H^2(S,\C) \}.
\end{align} Using this set of generators, one only needs to keep track of the intersection matrix of the following elements to evaluate descendent integrals.
\begin{align}
\label{classes}
    L_i \exp{\left(\frac{c_1(v)}{\text{rk}(v)} \right)},\quad\text{with }L_i\in H^2(S,\mathbb{C});\quad \exp{\left(\frac{c_1(v)}{\text{rk}(v)} \right)};\quad \boldsymbol{\operatorname{p}};\quad\frac{\boldsymbol{\operatorname{p}}}{\text{rk}(v)};\quad v.
\end{align}
One can calculate Mukai pairings between the above classes, the only interesting pairings are:
\begin{align*}
    (v,L_i \exp{\left(\frac{c_1(v)}{\text{rk}(v)} \right)})
    &=0\\
    (\boldsymbol{\operatorname{p}}/\text{rk}(v),L_i \exp{\left(\frac{c_1(v)}{\text{rk}(v)} \right)})&=0\\
    (L_i \exp{\left(\frac{c_1(v)}{\text{rk}(v)} \right)},L_j \exp{\left(\frac{c_1(v)}{\text{rk}(v)} \right)})
    &= L_iL_j\\
    (\exp{\left(\frac{c_1(v)}{\text{rk}(v)} \right)},L_i\exp{\left(\frac{c_1(v)}{\text{rk}(v)} \right)})
    &= 0\\
    (v,\exp{\left(\frac{c_1(v)}{\text{rk}(v)} \right)})
    &= \frac{D^2}{2r}-n,\\
    (\frac{\boldsymbol{\operatorname{p}}}{\text{rk}(v)},\exp{\left(\frac{c_1(v)}{\text{rk}(v)} \right)})
    &=-\frac{1}{r}\\
    (\exp{\left(\frac{c_1(v)}{\text{rk}(v)} \right)},\exp{\left(\frac{c_1(v)}{\text{rk}(v)} \right)})
    &=0\\
    (\boldsymbol{\operatorname{p}},v)&=-r
\end{align*}

Consider the following integral:
\[\int_{M(v)}P\left(\operatorname{ch}_i(L_j),\operatorname{ch}_l(\boldsymbol{\operatorname{p}}),\operatorname{ch}_k(\mathbbm{1}+\boldsymbol{\boldsymbol{\operatorname{p}}})\right),\] 
where $P$ is a polynomial, $L_j\in H^2(S,\mathbb{Q})$ and $v=(r,D,n)$. This integral is defined for general $v\in \Lambda$ that satisfies assumptions of Section \ref{sectionVirasoro}, one can now take a specific value by setting $v=\mathbbm{1}-(n-1)\boldsymbol{\operatorname{p}}$, which leads to the following

\begin{lemma}
\label{claim}
Let $P$ be a polynomial with variables in \eqref{basis}, and $v=(r,D,n)$
\begin{align}  
\begin{split}
\int_{M(v)}&P\left(\operatorname{ch}^\mathrm{H}_i(L_j),\operatorname{ch}^\mathrm{H}_l(\boldsymbol{\mathrm{p}}),\operatorname{ch}^\mathrm{H}_k(\mathbbm{1}+\boldsymbol{\mathrm{p}}) \right)\\
&=\int_{S^{[N]}}P\left(\operatorname{ch}^\mathrm{H}_i(L_j),r\operatorname{ch}^\mathrm{H}_l(\boldsymbol{\mathrm{p}}),\frac{1}{r}\operatorname{ch}^\mathrm{H}_k(\mathbbm{1}+\boldsymbol{\mathrm{p}}) \right),
\end{split}
\end{align}
where $2N-2=(v,v)=D^2-2rn$.
\end{lemma}

\begin{proof}
We now consider the RHS of \eqref{claim}. By setting $v=\mathbbm{1}-(N-1)\boldsymbol{\operatorname{p}}$, I have: 
\begin{align*}
    \begin{split}
    (\mathbbm{1}-(N-1)\boldsymbol{\operatorname{p}},\mathbbm{1}-(N-1)\boldsymbol{\operatorname{p}})=2N-2=(v,v)=D^2-2rn.
\end{split}
\end{align*}
The classes in \eqref{classes} become (up to multiplication of $1/r$ or $r$):
\begin{align}
\label{classes2}
    L_i,\quad\text{for }L_i\in H^2(S);\quad \frac{\mathbbm{1}}{r};\quad r\boldsymbol{\operatorname{p}};\quad\boldsymbol{\operatorname{p}};\quad \mathbbm{1}-(N-1)\boldsymbol{\operatorname{p}}.
\end{align}
Notice that $\mathbbm{1}$ is multiplied with $1/r$ and $\boldsymbol{\operatorname{p}}$ is multiplied with $r$, this is harmless since $r>0$. One can calculate Mukai pairings among these classes, the only interesting pairings are the following:
\begin{align*}
(L_i,L_j)&=L_iL_j\\
    (v,\frac{\mathbbm{1}}{r})&=\frac{1}{r}(N-1) = \frac{1}{r}(\frac{1}{2}D^2-rn)\\
    (\boldsymbol{\operatorname{p}},\frac{\mathbbm{1}}{r})&=-\frac{1}{r}\\
    (\boldsymbol{\operatorname{p}},v)&=-r.
\end{align*}
By comparing with the pairings of \eqref{classes} one can see that \eqref{classes} and \eqref{classes2} have the same pairing matrix.

By the last line of \eqref{calculation4.1}, the class $\boldsymbol{\operatorname{p}}$ in \eqref{classes} comes from the $B_1$ part of $\operatorname{ch}^\mathrm{H}_i(\boldsymbol{\operatorname{p}})$, it is multiplied by $r$ in \eqref{classes2}, therefore it comes from the $B_1$ part of $r\operatorname{ch}^\mathrm{H}_i(\boldsymbol{\operatorname{p}})$; analogously, $\mathbbm{1}/r$ in \eqref{classes2} comes from $1/r\operatorname{ch}^\mathrm{H}_i(\mathbbm{1}+\boldsymbol{\operatorname{p}})$.
Therefore, the classes in \eqref{classes2} are the arguments of $B$ and $B_1$ function of 
\[\{\operatorname{ch}_j^\mathrm{H}(L_i),r\operatorname{ch}^\mathrm{H}_k(\boldsymbol{\operatorname{p}}),\frac{1}{r}\operatorname{ch}_l^\mathrm{H}(\mathbbm{1}+\boldsymbol{\operatorname{p}})|j,k,l\in\mathbb{N}, \{L_i\}_i \text{ forms a basis of } H^2(S,\C) \}.\] 
Therefore by Theorem \ref{theoremUnivers}, the lemma is true.
\end{proof}
In the following paragraphs, I will use the above lemma to prove the Virasoro constraints on $M(v)$ in the positive rank case. Without of lose of generality one can always assume $D\in \mathbb{D}^S$ is a monomial and $D$ does not contain $\operatorname{ch}^\mathrm{H}_i(\mathbbm{1})$'s, since one can always use the relation 
\begin{align}
\label{relation}
    \operatorname{ch}_k^\mathrm{H}(\mathbbm{1}+\boldsymbol{\operatorname{p}})= \operatorname{ch}_k^\mathrm{H}(\mathbbm{1})+\operatorname{ch}_k^\mathrm{H}(\boldsymbol{\operatorname{p}})
\end{align}
 to rewrite $D$. Then by the Lemma \ref{claim}, one has 
\begin{align*}
    \int_{M(v)}R_k D =r^{N_p}\int_{S^{[N]}}R_k D,
\end{align*}
where $N_p$ is the number of $\operatorname{ch}^\mathrm{H}_i(\boldsymbol{\operatorname{p}})$ minus the number of $\mathrm{ch}^\mathrm{H}_j(\mathbbm{1} + \boldsymbol{\operatorname{p}})$ in $D$.  Next, I claim that 
\begin{align}
\label{claim2}
    \int_{M(v)}T_k D = r^{N_p} \int_{S^{[N]}}T_k D.
\end{align}

For $T_k$, the terms which are of interest are the terms containing $\operatorname{ch}_i(\mathbbm{1})$ or $\operatorname{ch}_i(\boldsymbol{\operatorname{p}})$, in $T_k$ those terms are:
\begin{align*}
\sum_{a+b=k}\left(a!b!\operatorname{ch}^\mathrm{H}_a(\mathbbm{1})\operatorname{ch}^\mathrm{H}_b(\boldsymbol{\operatorname{p}})+ a!b!\operatorname{ch}^\mathrm{H}_a(\boldsymbol{\operatorname{p}})\operatorname{ch}^\mathrm{H}_b(\mathbbm{1})\right).
\end{align*} and \begin{align*}
    \sum_{a+b=k} a!b!\operatorname{ch}^\mathrm{H}_a\operatorname{ch}^\mathrm{H}_b\left( 2\boldsymbol{\operatorname{p}}\right)
\end{align*}
All the contributions from $H^2(S,\mathbb{Q}) \otimes H^2(S,\mathbb{Q})$ cause no problem, since they only contains descendants of the form $\ch{i}{L}$ with $i\in \mathbb{N},L\in H^2(S,\Q)$ and therefore by lemma \ref{claim}, their contributions remain the same.

For the first term, one can use the relation \eqref{relation} to rewrite it as:
\begin{align}
\label{firstterm}
\sum_{a+b=k}&\bigg(a!b!\bigg[\operatorname{ch}_{a}^\mathrm{H}(\mathbbm{1}+\boldsymbol{\operatorname{p}})-\operatorname{ch}^\mathrm{H}_{a}(\boldsymbol{\operatorname{p}})\bigg]\operatorname{ch}^\mathrm{H}_b(\boldsymbol{\operatorname{p}})\\&\quad\quad\quad\quad+ a!b!\operatorname{ch}^\mathrm{H}_a(\boldsymbol{\operatorname{p}})\bigg[\operatorname{ch}^\mathrm{H}_b(\mathbbm{1}+\boldsymbol{\operatorname{p}})-\operatorname{ch}^\mathrm{H}_b(\boldsymbol{\operatorname{p}})\bigg]\bigg).
\end{align}
Using lemma \ref{claim} to pass from integral over $M(v)$ to $S^{[N]}$, \eqref{firstterm} becomes:
\begin{align}
\label{star}
     \sum_{a+b=k}&\bigg(a!b!\bigg[\frac{1}{r}\operatorname{ch}_{a}^\mathrm{H}(\mathbbm{1}+\boldsymbol{\operatorname{p}})-r\operatorname{ch}^\mathrm{H}_{a}(\boldsymbol{\operatorname{p}})\bigg]r\operatorname{ch}^\mathrm{H}_b(\boldsymbol{\operatorname{p}})\\ &\quad\quad\quad\quad + a!b!r\operatorname{ch}^\mathrm{H}_a(\boldsymbol{\operatorname{p}})\bigg[\frac{1}{r}\operatorname{ch}^\mathrm{H}_b(\mathbbm{1}+\boldsymbol{\operatorname{p}})-r\operatorname{ch}^\mathrm{H}_b(\boldsymbol{\operatorname{p}})\bigg]\bigg)\\
     &=\sum_{a+b=k} \bigg( 2a!b!(1-r^2)\operatorname{ch}^\mathrm{H}_{a}(\boldsymbol{\operatorname{p}})\operatorname{ch}^\mathrm{H}_{b}(\boldsymbol{\operatorname{p}}) \bigg)\\
    &\quad +\sum_{a+b=k}\bigg(a!b!\operatorname{ch}^\mathrm{H}_a(\mathbbm{1})\operatorname{ch}^\mathrm{H}_b(\boldsymbol{\operatorname{p}})+ a!b!\operatorname{ch}^\mathrm{H}_a(\boldsymbol{\operatorname{p}})\operatorname{ch}^\mathrm{H}_b(\mathbbm{1})\bigg). 
\end{align}
Next, consider the term $\sum_{a+b=k} a!b!\operatorname{ch}^\mathrm{H}_a\operatorname{ch}^\mathrm{H}_b\left( 2\boldsymbol{\operatorname{p}}\right)$, when passing from integral over $M(v)$ to $S^{[N]}$ using lemma \ref{claim}, this term becomes:
\begin{align}
\label{star2}
    \sum_{a+b=k} 2a!b!r^2&\operatorname{ch}^\mathrm{H}_a(\boldsymbol{\operatorname{p}})\operatorname{ch}^\mathrm{H}_b( \boldsymbol{\operatorname{p}})\\
    &=\sum_{a+b=k} 2a!b!\operatorname{ch}^\mathrm{H}_a(\boldsymbol{\operatorname{p}})\operatorname{ch}^\mathrm{H}_b( \boldsymbol{\operatorname{p}})+ \sum_{a+b=k} 2a!b!(r^2-1)\operatorname{ch}^\mathrm{H}_a(\boldsymbol{\operatorname{p}})\operatorname{ch}^\mathrm{H}_b( \boldsymbol{\operatorname{p}})
\end{align}
Therefore \eqref{star}, \eqref{star2} and the fact that other künneth component cause no problems imply \eqref{claim2}.

Now, let us look at the $S_k$ term.
The classes in $H^{0,q}(S)$ are $\mathbbm{1}\in H^{0,0}(S)$ and $\sigma\in H^{0,2}(S)$. There is only one such $\sigma$ up to scaling, since $\text{dim }H^{0,2}(S)=1$. Let us first look at the term 
\[ -\frac{1}{r} (k+1)! R_{-1}[\mathbbm{1}] \left(\operatorname{ch}_{k+1}^\mathrm{H}(\boldsymbol{\operatorname{p}}) D\right). \]
Evidently, when passing from integral over $M(v)$ to $S^{[N]}$, this term changes to 
\[-r^{N_p}(k+1)! R_{-1}[\mathbbm{1}] \left(\operatorname{ch}^\mathrm{H}_{k+1}(\boldsymbol{\operatorname{p}}) D\right).\]
Now consider the term 
\begin{align}
\label{sk}
    -\frac{1}{r} (k+1)! R_{-1}[\sigma] \left(\operatorname{ch}_{k+1}^\mathrm{H}(\overline{\sigma}) D\right), 
\end{align} 
where $\overline{\sigma}\cdot \sigma =\boldsymbol{\operatorname{p}}$. Notice that for $\delta\in H^\ast(S)$, $\sigma\cdot\delta\neq 0\implies \delta=\overline{\sigma} \text{ or } \mathbbm{1}$ or a linear combination of those. Because we assumed that $\operatorname{ch}_i(\mathbbm{1})$ does not appear in $D$, the terms that survive in \eqref{sk} are terms produced when $R_{-1}[\sigma]$ acts on different $\operatorname{ch}_{i}^\mathrm{H}(\overline{\sigma})$'s or on different $\operatorname{ch}_{i}^\mathrm{H}(\mathbbm{1}+\boldsymbol{\operatorname{p}})$'s. In the first case, it will create an additional $\operatorname{ch}_{i-1}(\boldsymbol{\operatorname{p}})$; in the second case, it will destroy an $\operatorname{ch}_{i}^\mathrm{H}(\mathbbm{1}+\boldsymbol{\operatorname{p}})$, in either case it will produce a factor $r$. Thus when passing from integral over $M(v)$ to $S^{[N]}$, \eqref{sk} changes to 
\[-r^{N_p}(k+1)! R_{-1}[\sigma] \operatorname{ch}^\mathrm{H}_{k+1}(\overline{\sigma}) D.\]
Therefore we find:
\begin{align*}
    \int_{M(v)}S_k D =r^{N_p}\int_{S^{[N]}}S_k D.
\end{align*}
Because the Virasoro constraints on $S^{[N]}$ are proven \cite{Moreira_2022}, the Virasoro constraints on $M(v)$ are also valid. This proves Theorem \ref{mainthm} for $\operatorname{dim} M>2$.

\subsection{$\operatorname{dim} M =2$ case}
As before, let $S$ be a non-singular projective K3 surface, let $v$ be a Mukai vector with $\operatorname{rk}(v)>0$.
In this case, I will recall some facts following Section 2.4 of \cite{Oberdieck_2022}. Let $M:=M_H(v)$ be a 2-dimensional moduli space of stable sheaves, hence it is a K3 surface. There is the following well known isometry of Mukai lattices:
\begin{align}
\widetilde{\Phi}:H^\ast(S,\mathbb{Q})\to H^\ast(M,\mathbb{Q}),\quad \gamma\mapsto\pi_{M\ast}(e^{-c_1(\mathcal{F})/rk(v)} v(F) \pi_S^\ast (\gamma)),
\end{align}
where $v(\mathcal{F}):= \operatorname{ch}(\mathcal{F})\sqrt{\operatorname{td}_{M\times S}}.$ The fact that $\widetilde{\Phi}$ is a Hodge isometry implies:
\begin{align*}
\widetilde{\Phi}(\boldsymbol{\operatorname{p}})=\mathrm{rk}(v)\mathbbm{1},\quad \widetilde{\Phi}(L)=\phi(L),\quad \widetilde{\Phi}(\mathbbm{1})=\frac{1}{\mathrm{rk}(v)} \boldsymbol{\operatorname{p}},
\end{align*}
where $\phi:H^2(S,\mathbb{Q})\to H^2(M,\mathbb{Q}) $ is a Hodge isometry, which means $\phi$ is an isometry with respect to the Mukai paring of $S$ and $M$, since $M$ is also a K3 surface.  

One has the following identities by direct calculation:
\begin{align}
    \operatorname{ch}_i^\mathrm{H}(\gamma) &= \bigg[\frac{1}{\sqrt{\operatorname{td}_M}}\widetilde{\Phi} \left( \frac{\gamma}{\sqrt{\operatorname{td}_S}} \right)\bigg]_{2i-p+q} \text{for }\gamma\in H^{p,q}(S).\\
    \text{ch}_i(\gamma) &= \bigg[\frac{1}{\sqrt{\operatorname{td}_M}}\widetilde{\Phi} \left( \frac{\gamma}{\sqrt{\operatorname{td}_S}} \right)\bigg]_{2\operatorname{deg}(\gamma)+2i-4}\text{for }\gamma\in H^{\ast}(S,\mathbb{Q}).
\end{align}
Here the new symbole $\operatorname{ch}_k(\gamma)$ has its geometric realization as:
\[\operatorname{ch}_k(\gamma) = \pi_{M\ast}\left( \operatorname{ch}_k(\mathcal{F}\otimes(\operatorname{det}\mathcal{F})^{-1/r})\cdot \pi_S^\ast \gamma \right).\]For $\gamma\in H^{p,q}(S)$ one has: 
\begin{align}
\label{rmkH}
    \ch{i}{\gamma}=\operatorname{ch}_{i+1+\frac{q-p}{2}}(\gamma).
\end{align}

Similarly as before:
\begin{align*}
\operatorname{ch}_i^\mathrm{H}(\mathbbm{1}+\boldsymbol{\operatorname{p}})= \bigg[\frac{1}{\sqrt{\operatorname{td}_M}}\widetilde{\Phi} \left( \frac{\mathbbm{1}+\boldsymbol{\operatorname{p}}}{\sqrt{\operatorname{td}_S}} \right)\bigg]_{2i} =\operatorname{ch}_i^\mathrm{H}(\mathbbm{1})+\operatorname{ch}_i^\mathrm{H}(\boldsymbol{\operatorname{p}}) .
\end{align*}

Because $\text{dim } M = 2$, one can explicitly calculate the cohomology class of $\operatorname{ch}^\mathrm{H}_i(\gamma)$. The non zero ones among them are:
\begin{align}
\begin{split}
    \operatorname{ch}_0^\mathrm{H}(\boldsymbol{\operatorname{p}}) &= \text{rk}(v)\mathbbm{1},\quad \\\operatorname{ch}_2^\mathrm{H}(\boldsymbol{\operatorname{p}})& = -\text{rk}(v)\boldsymbol{\operatorname{p}},\\
    \operatorname{ch}_2(\gamma)&=\phi(\gamma),\quad \text{where } \gamma\in H^{2}(S)\\
    \operatorname{ch}_2^\mathrm{H}(\mathbbm{1})&=-\text{rk}(v)\mathbbm{1} \\
    \operatorname{ch}_4^\mathrm{H}(\mathbbm{1})&=\left(\frac{1}{\text{rk}(v)}+\text{rk}(v)\right)\boldsymbol{\operatorname{p}}\\
    \operatorname{ch}^\mathrm{H}_2(\mathbbm{1}+\boldsymbol{\operatorname{p}})&=\frac{1}{\text{rk}(v)}\boldsymbol{\operatorname{p}}.
\end{split}
\end{align}

In fact, the geometric realization of $\ch{i}{\boldsymbol{\operatorname{p}}}$, $\ch{i}{\mathbbm{1}}$ are the same as $\operatorname{ch}_{i}(\boldsymbol{\operatorname{p}})$, $\operatorname{ch}_{i}(\mathbbm{1})$. 
 For $\gamma\in H^2(S,\C)$, using $\operatorname{ch}_2(\gamma)$ will make notation simpler since only $\operatorname{ch}_2(\gamma)\neq 0$ whereas there are potentially several $i$ such that $\ch{i}{\gamma}\neq 0$.

\begin{lemma}
\label{dimM=2}
Let $S$ be a non-singular and projective K3 surface, and let $v,M_H(v)$ be defined as above. Suppose $M_H(v)$ is 2 dimensional, and suppose $r:=\operatorname{rk}(v)>0$. Then
\[\int_{M} P\left(\operatorname{ch}^\mathrm{H}_i(\boldsymbol{\operatorname{p}}),\operatorname{ch}^\mathrm{H}_j(\gamma_l),\operatorname{ch}^\mathrm{H}_k(\mathbbm{1}+\boldsymbol{\operatorname{p}})\right) = \int_{S^{[1]}} P\left(r\operatorname{ch}^\mathrm{H}_i(\boldsymbol{\operatorname{p}}),\operatorname{ch}^\mathrm{H}_j(\gamma_l),\frac{1}{r}\operatorname{ch}^\mathrm{H}_k(\mathbbm{1}+\boldsymbol{\operatorname{p}})\right),\]
where $P$ is a polynomial and $\gamma_l\in H^2(S,\mathbb{Q})$. 
\end{lemma}
One can use \eqref{rmkH} to rewrite $\operatorname{ch}^\mathrm{H}_i(\gamma)$ as $\operatorname{ch}_2(\gamma)$. 
The only monomials whose integral on $M$ are non zero have one of the following three forms (up to multiplication by a constant):
\begin{align}
&\big[\operatorname{ch}_0^\mathrm{H}(\boldsymbol{\operatorname{p}})\big]^N\operatorname{ch}_2(\gamma_1) \operatorname{ch}_2(\gamma_2)\\
&\big[\operatorname{ch}^\mathrm{H}_0(\boldsymbol{\operatorname{p}})\big]^N\operatorname{ch}_2^\mathrm{H}(\boldsymbol{\operatorname{p}})\\
&\big[\operatorname{ch}_0^\mathrm{H}(\boldsymbol{\operatorname{p}})\big]^N\operatorname{ch}^\mathrm{H}_2(\mathbbm{1}+\boldsymbol{\operatorname{p}}),
\end{align}
where $N$ is some non-negative integer and $\gamma_1,\gamma_2\in H^2(S,\C)$. 
Their integral over $M$ are:
\begin{align}
\int_{M} \big[\operatorname{ch}_0^\mathrm{H}(\boldsymbol{\operatorname{p}})\big]^N\operatorname{ch}_2(\gamma_1) \operatorname{ch}_2(\gamma_2)&=r^N\cdot (\phi(\gamma_1),\phi(\gamma_2))\\
&=r^N\cdot (\gamma_1,\gamma_2)\\
\int_{M}\big[\operatorname{ch}_0^\mathrm{H}(\boldsymbol{\operatorname{p}})\big]^N\operatorname{ch}_2^\mathrm{H}(\boldsymbol{\operatorname{p}})&=-r^{N+1}\\
\int_{M}\big[\operatorname{ch}_0^\mathrm{H}(\boldsymbol{\operatorname{p}})\big]^N\operatorname{ch}_2^\mathrm{H}(\mathbbm{1}+\boldsymbol{\operatorname{p}})&=r^{N-1}.
\end{align}
One can also evaluate integrals on $S^{[1]}$:
\begin{align}
\int_{S^{[1]}} \big[r\operatorname{ch}_0^\mathrm{H}(\boldsymbol{\operatorname{p}})\big]^N\operatorname{ch}_2(\gamma_1) \operatorname{ch}_2(\gamma_2)&=r^N\cdot (\phi(\gamma_1),\phi(\gamma_2))\\
&=r^N\cdot (\gamma_1,\gamma_2)\\
\int_{S^{[1]}}\big[r\operatorname{ch}_0^\mathrm{H}(\boldsymbol{\operatorname{p}})\big]^N\big[r\operatorname{ch}_2^\mathrm{H}(\boldsymbol{\operatorname{p}})\big]&=-r^{N+1}\\
\int_{S^{[1]}}\big[r\operatorname{ch}_0^\mathrm{H}(\boldsymbol{\operatorname{p}})\big]^N\cdot\frac{1}{r}\operatorname{ch}^\mathrm{H}_2(\mathbbm{1}+\boldsymbol{\operatorname{p}})&=r^{N-1}.
\end{align}
By comparing the two evaluations, lemma \ref{dimM=2} is proven. Therefore one can use the same argument as in subsection \ref{bigger2} to show the Virasoro constraints are valid in case of $\text{dim } M=2$. Combined with the result of previous subsection, Theorem \ref{mainthm} is proven.

\section{Rank 0 case}
\label{rankzero}
The Virasoro operator in the section \ref{positiverank} only works for the strictly positive rank case, since the definition of $S_k$ has the constant $\frac{1}{r}$. To solve this, one needs the concept of $\delta$-normalized universal sheaf \cite{bojko2022virasoro} which will be presented in the following subsection.  
\subsection{$\delta$-normalized universal sheaf and the invariant Virasoro operator}

 Let me recall the notion of $\delta$-normalized universal sheaf in \cite{bojko2022virasoro}. Let $v\in \Lambda$ $M:=M_H(v)$ and $\pi_M,\pi_S$ being defined as above. And I assume that there exists a universal sheaf on $M\times S$. Let $\alpha\in K_0(X)_{\mathbb{Q}}$ be the topological type of sheaves that $M$ parameterizes, i.e. one has $\operatorname{ch}(\alpha)\cdot \sqrt{\text{td}_S} = v$.
Suppose $\delta\in \mathrm{H}^\bullet(S,\mathbb{Z})$ is an algebraic class such that $\int_S \delta\cdot \operatorname{ch}(\alpha)\neq 0$. We say a universal sheaf $\mathcal{G}$ is $\delta$-normalized if 
\[\xi_{\mathcal{G}}(\operatorname{ch}_1^\mathrm{H}(\delta))=0.\] 
By the remark 2.14 of \cite{bojko2022virasoro}, the $\delta$-normalized universal sheaf always exists and is unique as an element of the rational K-theory of $M\times X$.
Let $\mathcal{F}$ be any universal sheaf, its unique $\delta$-normalized universal sheaf is:
\begin{align}
    \mathcal{F}_\delta = \mathcal{F}\otimes e^{-\xi_{\mathcal{F}}(\operatorname{ch}^\mathrm{H}_1(\delta))/\int_S \delta\cdot \operatorname{ch}(\alpha)}.
\end{align}
Also by the same remark, $\xi_{\mathbb{G}_\delta}=\xi_{\mathbb{G}}\circ \eta$ where $\eta:\mathbb{D}^S\to\mathbb{D}^S$ is defined as:
\[\eta = \sum_{j\geq 0} \left( -\frac{\ch{1}{\delta}}{\int_S\delta\cdot\operatorname{ch}(\alpha)} \right)^j R^j_{-1}.\]Therefore, the geometric realization with respect to $\delta$-normalized sheaf is still valid even such a sheaf does not exist in usual sense. By direct calculation one can check that $(\mathcal{F}\otimes\mathcal{L})_\delta =\mathcal{F}_\delta $ for any line bundle $\mathcal{L}$ on $M$. Since we want that the geometric realisation of any $D\in \mathbb{D}^S$ is independent of choice of the universal sheaf, $\delta$-normalised universal sheaf will be a good choice to formulate the Virasoro constraint. For rank zero sheaves, $\boldsymbol{\operatorname{p}}$-normalised universal sheaf is clearly not possible, since $\int_S \boldsymbol{\operatorname{p}}\cdot \operatorname{ch}(\alpha)=0$. A good choice, in this case, would be a $Y$-normalised universal sheaf, for some $Y\in H^{1,1}(S)$, which will be used in the next subsection.

Now, as in previous sections, let $v=(r,d,n)$ with $r>0$. 
For later use, let me introduce the following notations:
\begin{align}
    L_k&:=R_k+T_k\\
    S_{k,0}&:=-\frac{(k+1)!}{r}R_{-1}\operatorname{ch}_{k+1}^\mathrm{H}(\boldsymbol{\operatorname{p}})\\
    S_{k,2}&:=-\frac{(k+1)!}{r}R_{-1}[\sigma]\operatorname{ch}_{k+1}^\mathrm{H}(\overline{\sigma}), 
\end{align}
Where $\sigma\in H^{0,2}(S)$ with $S$ a K3 surface and $\sigma\cdot\overline{\sigma}=\boldsymbol{\operatorname{p}}$.
As before, for a $D\in\mathbb{D}^S$, \[S_{k,0}(D)=-\frac{(k+1)!}{r}R_{-1}(\operatorname{ch}_{k+1}^\mathrm{H}(\boldsymbol{\operatorname{p}})D),\] similar for $S_{k,2}$.
Also, one has $S_k=S_{k,0}+S_{k,2}$.

\begin{proposition}
As before, let $\mathcal{L}_k=R_k+T_k+S_k$. The commutation relation
    $[R_{-1},\mathcal{L}_k]=(k+1) \mathcal{L}_{k-1}$ is true.
\end{proposition}

\begin{proof}
    The relation $[R_{-1},L_k]=(k+1) L_{k-1}$ is already checked in \cite{bojko2022virasoro}. Let us calculate $[R_{-1},S_k]$, using $\sum _{t}\delta_t^L\otimes \delta_t^R$ as the K\"unneth decomposition of $\Delta_\ast \mathbbm{1}\in H^\ast(S\times S)$ :
\begin{align*}
    [R_{-1},S_k] D &= R_{-1}S_kD-S_kR_{-1}D\\
    &=-\frac{(k+1)!}{r}\sum_{\{t|p^L_t=0\}}R_{-1}[\delta^L_t]\bigg(\operatorname{ch}_{k}^\mathrm{H}(\delta^R_t)D+ \operatorname{ch}_{k+1}^\mathrm{H}(\delta^R_t)R_{-1}D \bigg) \\
    &\quad\quad-S_kR_{-1}D\\
    &=-\frac{(k+1)!}{r}\sum_{\{t|p^L_t=0\}}R_{-1}[\delta^L_t]\bigg(\operatorname{ch}_{k}^\mathrm{H}(\delta^R_t)D \bigg)\\
    &=(k+1)S_{k-1}D,
\end{align*}
where in the second line, I used that $R_{-1}$ commutes with $R_{-1}[\gamma]$.
Therefore I have $[R_{-1},\mathcal{L}_k]=(k+1) \mathcal{L}_{k-1}$. 
\end{proof}

Define the following invariant Virasoro operator as in \cite{bojko2022virasoro}.
\begin{align}
\label{LinvDef}
    \mathcal{L}_\text{inv}:=\sum_{j\geq -1}\frac{(-1)^j}{(j+1)!}\mathcal{L}_jR^{j+1}_{-1}
\end{align}
Using the commutation relation\begin{align}
    [R_{-1},\mathcal{L}_k]=(k+1)\mathcal{L}_{k-1}
\end{align} 
one has:
\begin{align}
    R_{-1}\mathcal{L}_\text{inv} = \sum_{j\geq -1}\frac{(-1)^j}{(j+1)!}\mathcal{L}_jR^{j+2}_{-1}+ \sum_{j\geq -1}\frac{(-1)^j}{(j+1)!}(j+1)\mathcal{L}_{j-1}R^{j+1}_{-1}=0.
\end{align}
Therefore the geometric realization of $\mathcal{L}_\text{inv}D, D\in \mathbb{D}^S$ does not depend on the choice of universal sheaf, by lemma 2.8 of \cite{bojko2022virasoro}. More precisely, this means the following: given two universal sheaf $\mathcal{F}$ and $\mathcal{F}'=\mathcal{F}\otimes \pi_M^\ast L$ where $L$ is a line bundle on $M$, then 
\[\xi_\mathcal{F}(\mathcal{L}_{\text{inv}}D)= \xi_{\mathcal{F}'}(\mathcal{L}_{\text{inv}}D)\quad \forall D\in \mathbb{D}^S.\]  
Therefore I could omit the geometric realization and write 
\[\int_M \mathcal{L}_{\text{inv}}D.\]

In the positive rank case, I have used $\mathcal{F}\otimes \det(\mathcal{F})^{-1/r}$ in the geometric realization, which is not a universal sheaf. One can adapt lemma 2.19 of \cite{bojko2022virasoro} to show that the Virasoro constraint is equivalent when using $\mathcal{F}\otimes \det(\mathcal{F})^{-1/r}$ or $\mathcal{F}\otimes \det(\mathcal{F})^{-1/r}\otimes\pi_S^\ast\Delta^{1/r}$, where $\Delta\in \mathrm{Pic}(S)$ such that $c_1(\Delta)=c_1(v)$. Also one can check that the latter sheaf is a $\boldsymbol{\operatorname{p}}$-normalised universal sheaf. 

\begin{proposition}
\label{5.2}
    \[\int_{M} \xi_{\mathcal{F}\otimes \det(\mathcal{F})^{-1/r}} \left(\mathcal{L}_k D\right) = 0 \quad\forall k\geq -1, \forall D \in \mathbb{D}^S\]
    \[\iff\int_{M} \xi_{\mathcal{F}\otimes \det(\mathcal{F})^{-1/r}\otimes\pi_S^\ast\Delta^{1/r}} \left(\mathcal{L}_k D\right) = 0 \quad\forall k\geq -1, \forall D \in \mathbb{D}^S,\]
    where $\Delta\in \operatorname{Pic}(S)$ such that $c_1(\Delta)=c_1(v)$. Also $\mathcal{F}\otimes \det(\mathcal{F})^{-1/r}\otimes\pi_S^\ast\Delta^{1/r}$ is a $\boldsymbol{p}$-normalised sheaf.
\end{proposition}
\begin{proof}
    One need to show that in K3 case, the algebra isomorphism $\operatorname{F}:\mathbb{D}^S\to \mathbb{D}^S$ in lemma 2.19 of \cite{bojko2022virasoro} satisfies $\mathcal{L}_k\circ \operatorname{F} = \operatorname{F}\circ \mathcal{L}_k$ for all $k\geq -1$. More precisely: recall that $\operatorname{F}$ is defined as $\operatorname{F}\left(\operatorname{ch}_i^{\mathrm{H}}(\gamma)\right)=\operatorname{ch}_i^{\mathrm{H}}\left(e^{c_1(L)} \gamma\right)$ with $L\in \operatorname{Pic}(S)$, and It is already shown in \cite{bojko2022virasoro} that $L_k\circ F = F \circ L_k$ for $k\geq -1$. I need to show $\operatorname{F}\circ S_k=S_k\circ \operatorname{F}$ in addition. This is true since $\operatorname{F}$ commutes with $R_{-1}$, $R_{-1}[\gamma]$ and 
\begin{align}
\label{discussion}
\begin{split}
    \operatorname{F} S_k D &= -\frac{(k+1)!}{r}\operatorname{F}\bigg(R_{-1}[\sigma]\left(\operatorname{ch}_{k+1}^\mathrm{H}(\overline{\sigma})D\right) + R_{-1}\left(\operatorname{ch}_{k+1}^\mathrm{H}(\boldsymbol{\operatorname{p}})D\right)\bigg)\\
    &=-\frac{(k+1)!}{r}\bigg(R_{-1}[\sigma]\left(\operatorname{F}\operatorname{ch}_{k+1}^\mathrm{H}(\overline{\sigma})D\right) + R_{-1}\left(\operatorname{F}\operatorname{ch}_{k+1}^\mathrm{H}(\boldsymbol{\operatorname{p}})D\right)\bigg)\\
    &=-\frac{(k+1)!}{r}\bigg(R_{-1}[\sigma]\left(\operatorname{ch}_{k+1}^\mathrm{H}(\overline{\sigma})\operatorname{F}D\right) + R_{-1}\left(\operatorname{ch}_{k+1}^\mathrm{H}(\boldsymbol{\operatorname{p}})\operatorname{F}D\right)\bigg)\\
    &= S_k \operatorname{F} D,
    \end{split}
\end{align}
where $\sigma$ is the only class in $H^{0,2}(S)$ and $e^{c_1(H)}\cdot \overline{\sigma} = \overline{\sigma}$ is used.
\end{proof}

Consider Mukai vector $v=(n,d,0)$ with $n,d\neq 0$. I want to find a Virasoro operators $\mathcal{L}^Y_k=R_k+T_k+S^Y_k$ which satisfy the Virasoro constraints when using geometric realization with a $Y$-normalised universal sheaf $\mathcal{F}_Y$, where $Y\in H^{1,1}(S)$.
This means I want to find Virasoro operators $\mathcal{L}^Y_k$ which satisfy:
\begin{align}
\label{wanttosatisfy}
    \int_{M(n,d,0)}\xi_{\mathcal{F}_Y}(\mathcal{L}^Y_kD)=0,\quad\forall D\in\mathbb{D}^S,\forall k\geq -1.
\end{align}
Let me define the $S^Y_k$ operator as follows:
\begin{align}
\label{SY}
\begin{split}
    S^Y_{k,0}&:=-\frac{(k+1)!}{\int_S Y\cdot \operatorname{ch}(\alpha)}R_{-1}\operatorname{ch}_{k+1}^\mathrm{H}(Y)\\
S^Y_{k,2}&:=- \frac{(k+1)!}{n}R_{-1}[\sigma]\operatorname{ch}_{k+1}^\mathrm{H}(\overline{\sigma})\\
    S^Y_k &:= S^Y_{k,0}+S^Y_{k,2},
\end{split}
\end{align}
where $\sigma\in H^{0,2}(X)$, $\operatorname{ch}(\alpha)\cdot \sqrt{\text{td}_S} = v$ and $\overline{\sigma}\sigma=\boldsymbol{\operatorname{p}}$. It turns out that this is the correct definition of $S^Y_k$ in order for the Virasoro operators to satisfy the desired property \eqref{wanttosatisfy}. I will prove this statement in the following paragraphs.
\begin{proposition}
\label{equi}
Let $v=(n,d,r)$ with $n,d \neq 0$ be a Mukai vector, and let $Y\in H^{1,1}(S)$, then:
\[\int_{M(v)}\xi_{\mathcal{F}_Y}(\mathcal{L}^Y_k D) = 0 \text{ for any $k\geq -1$, $D\in \mathbb{D}^S$} 
    \label{Linv}\]
\[\Longleftrightarrow\int_{M(v)}\mathcal{L}_{\mathrm{inv}} D = 0.\]
\end{proposition}

To prove this proposition, I need the following
\begin{claim}
\label{claimcommuteJ}
    $\mathcal{L}^Y_k(JD)=J\mathcal{L}^Y_k D$ where $J=\operatorname{ch}_1^{\mathrm{H}}(Y) \in \mathbb{D}^S$. 
\end{claim}
\begin{proof}[proof of claim]
\begin{align}
    R_k(J D)&= (k+1)!\operatorname{ch}_{k+1}^{\mathrm{H}}(Y)D+JR_k D\\
    T_k(JD)&=JT_kD \\
    S^Y_{k,0}(JD) &= -\frac{(k+1)!}{\int_S Y\cdot \operatorname{ch}(\alpha)}R_{-1}\left(\operatorname{ch}_{k+1}^\mathrm{H}(Y) JD\right)\\
    &=JS^Y_{k,0}(D)-\frac{(k+1)!}{\int_S Y\cdot \operatorname{ch}(\alpha)}\operatorname{ch}_{k+1}^\mathrm{H}(Y) \operatorname{ch}_0^{\mathrm{H}}(Y)D \\
    &=JS^Y_{k,0}(D) - (k+1)!\operatorname{ch}_{k+1}^{\mathrm{H}}(Y)D\\
    S^Y_{k,2}(JD)&= - \frac{(k+1)!}{n}R_{-1}[\sigma]\left(\operatorname{ch}_{k+1}^\mathrm{H}(\overline{\sigma}) JD\right)\\
    &=JS^Y_{k,2}(D) \text{  since $R_{-1}[\gamma]J=0$}.
\end{align}
Summing all terms gives the claim.
\end{proof}

\begin{proof}[Proof of proposition]
This proof of the Proposition \ref{equi} adapts the proof of the Proposition 2.16 of \cite{bojko2022virasoro}. To be more complete, I write the entire proof. First note:
\begin{align}
\begin{split}
\label{null1}
    \sum_{j \geqslant-1} & \frac{(-1)^j}{(j+1) !} S^Y_{j,0} R_{-1}^{j+1} \\
& =-\frac{1}{\int_S Y \cdot \operatorname{ch}(\alpha)} \sum_{j \geqslant-1}(-1)^j\left(\operatorname{ch}_j^{\mathrm{H}}(Y) R_{-1}^{j+1}+\operatorname{ch}_{j+1}^{\mathrm{H}}(Y) R_{-1}^{j+2}\right)=0,
\end{split}
\end{align}
where the convention $\mathrm{ch}_{i<0}^H(Y) = 0$ is used as before.
Analogously, one also has 
\begin{equation}
\label{null2}
    \sum_{j \geqslant-1} \frac{(-1)^j}{(j+1) !} S_{j,0} R_{-1}^{j+1}=0.
\end{equation}
Therefore
\begin{align}
\label{Linv}
\begin{split}
    \mathcal{L}_{\mathrm{inv}}&=\sum_{j \geqslant-1} \frac{(-1)^j}{(j+1) !} (L_j+ S_{j,0}+S_{j,2})R_{-1}^{j+1}\\
    &=\sum_{j \geqslant-1} \frac{(-1)^j}{(j+1) !} (L_j+ S^Y_{j,0}+S^Y_{j,2})R_{-1}^{j+1}\\
    &=\sum_{j \geqslant-1} \frac{(-1)^j}{(j+1) !} \mathcal{L}^Y_jR_{-1}^{j+1},
    \end{split}
\end{align}
where the second equality uses \eqref{null1} and \eqref{null2} and $S^Y_{k,2}=S_{k,2}$.
Therefore $\Longrightarrow$ is obvious.

Let us prove $\Longleftarrow$ next. The idea is to prove by induction, suppose the first part of the proposition holds for every $k'>k$ (it certainly holds for $k'>\operatorname{dim}(M)$ for degree reasons).
Let $J=\operatorname{ch}_1^{\mathrm{H}}(Y) \in \mathbb{D}^S$ as before, then the fact that $\mathcal{F}_Y$ is $Y$-normalised implies: $\xi_{\mathcal{F}_Y}(J)=0$.
Applying $\mathcal{L}_{\mathrm{inv}}$ in \eqref{Linv} to $J^{k+1}D$ gives
\begin{align}
\label{eqinvk}
    0=\int_{M(v)} \xi_{\mathcal{F}_Y}\left(\mathcal{L}_{\mathrm{inv}}\left(J^{k+1} D\right)\right)=\sum_{j \geqslant-1} \frac{(-1)^j}{(j+1) !} \int_{M(v)} \xi_{\mathcal{F}_Y}\left(\mathcal{L}_j^YR_{-1}^{j+1}\left(J^{k+1} D\right)\right)
\end{align}
Since $R_{-1}$ is a derivation and $R_{-1}(J)=\operatorname{ch}_0^{\mathrm{H}}(Y)=\int_S Y \cdot \operatorname{ch}(\alpha)$, I have
\begin{align}
R_{-1}^{j+1}\left(J^{k+1} D\right) & =\sum_{s=0}^{j+1}\binom{j+1}{s}R_{-1}^s\left(J^{k+1}\right) R_{-1}^{j+1-s}(D) \\
\label{eqinvk2}
& =\sum_{s=0}^{\min (k+1, j+1)}\binom{j+1}{s}\frac{(k+1) !}{(k+1-s) !} z^s J^{k+1-s} R_{-1}^{j+1-s}(D),
\end{align}
where I denote  $z=\int_S \delta \cdot \operatorname{ch}(\alpha)$. Using the fact that $\mathcal{L}_k^Y$ commutes with $J$, i.e. the Claim \ref{claimcommuteJ} and $\xi_{\mathcal{F}_Y}(J)=0$, one has that the only non vanishing terms in \eqref{eqinvk} come from the terms with $k=s-1$ in \eqref{eqinvk2}. Therefore:
\begin{align}
    0=\sum_{j \geqslant k} \frac{(-1)^j}{(j-k) !} z^{k+1} \int_{M(v)} \xi_{\mathcal{F}_Y}\left(\mathcal{L}^Y_j R_{-1}^{j-k}(D)\right).
\end{align}
By induction hypothesis, all terms with $j>k$ vanishes, therefore one has:
\begin{align}
    0=\int_{M(v)} \xi_{\mathcal{F}_Y}\left(\mathcal{L}_k^YD\right).
\end{align}
This finishes the proof. 
\end{proof}

In section 3, I have proven:
\[\int_{M(n,d,0)}\xi_{\mathcal{F}\otimes(\operatorname{det}\mathcal{F})^{-1/n}}(\mathcal{L}_kD)=0,\quad\forall D\in\mathbb{D}^S,\forall k\geq -1, \]
which, by the Proposition \ref{5.2}, implies 
\[\int_{M(n,d,0)}\xi_{\mathcal{F}_{\boldsymbol{\operatorname{p}}}}(\mathcal{L}_kD)=0,\quad\forall D\in\mathbb{D}^S,\forall k\geq -1.\]
This implies, by the definition \eqref{LinvDef} of $\mathcal{L}_\text{inv}$, that: 
\[\int_{M(n,d,0)}\xi_{\mathcal{F}_{\boldsymbol{\operatorname{p}}}}(\mathcal{L}_\text{inv}D)=\int_{M(n,d,0)}\mathcal{L}_\text{inv}D=0,\quad\forall D\in\mathbb{D}^S.\]
Above, by the Proposition \ref{equi}, implies:
\[\int_{M(n,d,0)}\xi_{\mathcal{F}_Y}(\mathcal{L}^Y_kD)=0,\quad\forall D\in\mathbb{D}^S,\forall k\geq -1.\]

\subsection{Using Markman operator to deduce Virasoro for rank zero case.}
Now write the sheaf $\mathcal{F}_Y$ in terms of the language of \cite{Oberdieck_2022}:
\begin{align}
    \mathcal{F}_Y &= \mathcal{F}\otimes e^{-\xi_{\mathcal{F}}(\operatorname{ch}^\mathrm{H}_1(Y))/\int_S Y\cdot \operatorname{ch}(\alpha)}\\
    &=\mathcal{F}\otimes \exp{\left(\frac{\theta_{\mathcal{F}}(Y)}{(Y,v)}\right)}\\
    &=\mathcal{F}\otimes \exp{\left(\frac{\theta_{\mathcal{F}}(v)}{(v,v)}\right)}\otimes\exp{\left(\frac{\theta_{\mathcal{F}}(Y)}{(Y,v)}-\frac{\theta_{\mathcal{F}}(v)}{(v,v)}\right)},
\end{align}
where those are identities in $K$-group of coherent sheaves on $M \times S$.
Therefore, by similar calculations as \eqref{calculation4.1}, one has:
\begin{align}
\label{70}
    \xi_{\mathcal{F}_Y}(\operatorname{ch}^\mathrm{H}_\bullet(\gamma))=\exp{\left(  B_1\left( \frac{Y}{(Y,v)}-\frac{v}{(v,v)} \right) \right)}B\left( \gamma^\vee \sqrt{\text{td}_S}^{-1} \right).
\end{align}
Therefore in order to use the Markmann operator one only needs to keep track of intersection matrix of $(v,\boldsymbol{\operatorname{p}},\mathbbm{1},L_i)$, where $L_i\in H^2(S)$ and $Y=L_i$ for one of the $i$'s. 

Consider another Mukai vector $v'=(0,d,n)$, $\text{Span}_{\mathbb{C}}(v,\boldsymbol{\operatorname{p}},\mathbbm{1},L_i)$ and $\text{Span}_{\mathbb{C}}(v',\mathbbm{1},\boldsymbol{\operatorname{p}},L_i)$ with $L_i\in H^2(X)$ are non-degenerate with respect to the Mukai pairing. Then $(v,\boldsymbol{\operatorname{p}},\mathbbm{1},L_i)$ and $(v',\mathbbm{1},\boldsymbol{\operatorname{p}},L_i)$ have the same intersection matrix, therefore by Lemma \ref{Oberdlemma}, there is an isometry sending $v\mapsto v'$, $\mathbbm{1}\mapsto \boldsymbol{\operatorname{p}}$, $\boldsymbol{\operatorname{p}}\mapsto\mathbbm{1}$ and $L_i\mapsto L_i$. Thus let us define operator $\boldsymbol{\mathrm{M}}:\mathbb{D}^S\to\mathbb{D}^S $ which is an algebra isomorphism as 
\begin{align}
    \boldsymbol{\mathrm{M}} (\operatorname{ch}_i^\mathrm{H}(\mathbbm{1}))= \operatorname{ch}_i^\mathrm{H}(-\mathbbm{1})\\
    \boldsymbol{\mathrm{M}} (\operatorname{ch}_i^\mathrm{H}(\boldsymbol{\operatorname{p}}))= \operatorname{ch}_i^\mathrm{H}(\mathbbm{1}+\boldsymbol{\operatorname{p}})\\
     \boldsymbol{\mathrm{M}} (\operatorname{ch}_i^\mathrm{H}(L_i))= \operatorname{ch}_i^\mathrm{H}(L_i)
\end{align}
The operator $\boldsymbol{\mathrm{M}}$ is defined to make the intersection matrix of the arguments of function $B(\bullet)$ in \eqref{70} invariant after the isometry.

In resume, the following constraint on $M(v')$ is proven by the Theorem \ref{mainthm} and section 5.1: 
\begin{align}
\label{72}
    \int_{M({v'})} \xi_{\mathcal{F}_Y}\left(\boldsymbol{\mathrm{M}}\circ\mathcal{L}_k^Y D\right)=0,\quad \forall D\in\mathbb{D}^S,\forall k\geq -1. 
\end{align}
We also need the following lemma to transform \eqref{72} in to a nicer form.
\begin{lemma}
\label{claimM}
$\boldsymbol{\mathrm{M}}\circ\boldsymbol{\mathrm{M}}=\operatorname{id}$, 
    $\boldsymbol{\mathrm{M}}\circ R_k = R_k \circ \boldsymbol{\mathrm{M}} $ and $\boldsymbol{\mathrm{M}}\circ T_k = T_k \circ \boldsymbol{\mathrm{M}}$.
\end{lemma}
\begin{proof}[Proof of lemma]
    $\boldsymbol{\mathrm{M}}\circ\boldsymbol{\mathrm{M}}=\operatorname{id}$ and $\boldsymbol{\mathrm{M}}\circ R_k = R_k \circ \boldsymbol{\mathrm{M}} $ are obvious. 
    \begin{align*}
        &\boldsymbol{\mathrm{M}}T_k  \\
        &= \boldsymbol{\mathrm{M}}\left( \sum_{i+j=k}(-1)^{2-p^L}i!j!\operatorname{ch}_{i}^\mathrm{H}\operatorname{ch}_{j}^\mathrm{H}(\text{td}_S)\right)\\
        &=\boldsymbol{\mathrm{M}}\left( \sum_{\substack{i+j=k\\\gamma_i^L,\gamma_i^R\in H^{2}(S)}} (-1)^{2-p^L} i!j! \operatorname{ch}^\mathrm{H}_i(\gamma_i^L)\operatorname{ch}^\mathrm{H}_j(\gamma_i^R) + \sum_{i+j=k}i!j!\left( \operatorname{ch}^\mathrm{H}_i(\mathbbm{1})\operatorname{ch}^\mathrm{H}_j(\boldsymbol{\operatorname{p}}) +\operatorname{ch}^\mathrm{H}_i(\mathbbm{1})\operatorname{ch}^\mathrm{H}_j(\boldsymbol{\operatorname{p}})\right) \right)\\
        &\quad\quad+\boldsymbol{\mathrm{M}}\left( 2\sum_{i+j=k}i!j! \operatorname{ch}^\mathrm{H}_i(\boldsymbol{\operatorname{p}})\operatorname{ch}^\mathrm{H}_j(\boldsymbol{\operatorname{p}}) \right)\\
        &=\sum_{\substack{i+j=k\\\gamma_i^L,\gamma_i^R\in H^{2}(S)}} (-1)^{2-p^L} i!j! \operatorname{ch}^\mathrm{H}_i(\gamma_i^L)\operatorname{ch}^\mathrm{H}_j(\gamma_i^R) \\
        &\quad\quad+ \sum_{i+j=k} i!j! \left[ -\operatorname{ch}^\mathrm{H}_i(\mathbbm{1})\operatorname{ch}^\mathrm{H}_j(\boldsymbol{\operatorname{p}})- \operatorname{ch}^\mathrm{H}_i(\boldsymbol{\operatorname{p}})\operatorname{ch}^\mathrm{H}_j(\mathbbm{1}) -2 \operatorname{ch}^\mathrm{H}_i(\mathbbm{1})\operatorname{ch}^\mathrm{H}_j(\mathbbm{1})\right]\\
        &\quad\quad +2\sum_{i+j=k}i!j!\left[ \operatorname{ch}^\mathrm{H}_i(\mathbbm{1})\operatorname{ch}^\mathrm{H}_j(\mathbbm{1}) +\operatorname{ch}^\mathrm{H}_i(\mathbbm{1})\operatorname{ch}^\mathrm{H}_j(\boldsymbol{\operatorname{p}})+ \operatorname{ch}^\mathrm{H}_i(\boldsymbol{\operatorname{p}})\operatorname{ch}^\mathrm{H}_j(\mathbbm{1}) + \operatorname{ch}^\mathrm{H}_i(\boldsymbol{\operatorname{p}})\operatorname{ch}^\mathrm{H}_j(\boldsymbol{\operatorname{p}})\right]\\
        &=T_k.
    \end{align*}
\end{proof}

Therefore the following constraint, in the rank zero case, is true:
\begin{proposition}
\label{rankzeroprop}
Let $v=(0,d,n)$ with $d,n>0$ be a Mukai vector, let $Y\in H^{1,1}(S)$ and let $S^Y_k:=S^Y_{k,0}+S^Y_{k,2}$ as defined in \eqref{SY}, then the following constraints are true:
\[
    \int_{M(v)} \xi_{\mathcal{F}_Y}\bigg(\left(R_k+T_k+ \boldsymbol{\mathrm{M}}\circ S^Y_k\circ\boldsymbol{\mathrm{M}}
    \right)D\bigg) =0\quad \forall D\in \mathbb{D}^S,\forall k\geq -1.
\]
\end{proposition}

The lemma \ref{claimM} means that the operator $\boldsymbol{\mathrm{M}}$ keeps the operator $L_k=R_k+T_k$ invariant, which justify to call above constraint the Virasoro constraint.

\begin{remark}
Assume $F\in M_{(0,d,n)} $ with $n>0$.
    On the set level, the map $F\mapsto F^D:=\ext(F,\omega_S)$ is a bijection, notice that $\ext(F,\omega_S)\cong\mathcal{R}\hom(F,\mathcal{O}_S)[1]=F^\vee[1]$. The type of $F^\vee$ is $\operatorname{ch}(F^\vee[1])=(0,d,-n)$, therefore one may expect that the proposition \ref{rankzero} also holds for spaces $M_{(0,d,-n)}$ with $d,n>0$.
\end{remark}

\section{Negative Virasoro operators}

In this section, let $S$ be a K3 surface as before. In \cite{bojko2022virasoro}, the Virasoro operators are related with lattice vertex algebra operators. Inspired by this correspondence, I will construct the operators $L_{k<-1}$ which along with previously defined $L_{k\geq -1}$ form a Virasoro algebra with the central charge $e(S)=24$, the topological Euler characteristic of $S$.

Let $k>1$ define the following derivations:

\begin{equation}
  \forall \gamma\in H^\ast(S,\C),\quad  R_{-k} \operatorname{ch}^\mathrm{H}_i(\gamma)=
    \begin{cases}
      \frac{(i-k)!}{(i-1)!}\operatorname{ch}^\mathrm{H}_{i-k}(\gamma), & \text{if}\ i\geq k \\
      0, & \text{otherwise}
    \end{cases}
\end{equation}
Define the operator $\operatorname{d}_i[v]$ acting on $D\in \mathbb{D}^S$ as a derivation as follows:
\begin{align}
   \forall i>0,j\geq 0,\quad \operatorname{d}_i[v] \operatorname{ch}^\mathrm{H}_j(\gamma)=\frac{\delta_{ij}}{(i-1)!}\int_S \gamma\cdot v \quad\forall \gamma\in H^\ast(S,\C).
\end{align}
for $v$ a homogeneous element of $H^\ast(S,\mathbb{C})$. 
Let $\sum_i\delta_i^L\otimes \delta_i^R$ be the Künneth decomposition of the diagonal class $\Delta\in H^\ast(S\times S,\C)$ and $\delta_i^L,\delta_j^R$ have Hodge type $(p^L_i,q^L_i)$ and $(p^R_i,q^R_i)$, respectively.
Next, choose such a Künneth decomposition $\sum_w\delta_w^L\otimes \delta_w^R$, define the operator $T_{-k}$ as\footnote{For surfaces with only even cohomology, $(-1)^{\operatorname{dim}(S)-p^L}=(-1)^{p^Lp^R}$. The latter format is used in \cite{Moreira_2022},}:
\begin{align}
    T_{-k}= -\frac{1}{4}\sum_{w} \sum_{\substack{i+j=k\\i>0\\j>0}}(-1)^{p_w^Lp_w^R}\operatorname{d}_j[\delta^L_w]\operatorname{d}_i[\delta^R_w] + \left(\frac{1}{4}\int_S \operatorname{td}_S\right)\sum_{\substack{i+j=k\\i>0\\j>0}}  \operatorname{d}_j[\boldsymbol{\operatorname{p}}]\operatorname{d}_i[\boldsymbol{\operatorname{p}}].
\end{align}
By the same reason why $T_k,k\geq -1$ are independent of the representations of the Künneth decomposition, which is since $T_k$ was defined in a multi-linear way, one can see that $T_{-k},k>0$ are also independent of the Künneth decomposition.

To clarify calculations, I will choose a particular representation of Künneth decomposition. Let $\{\alpha_i \}_i$ be a basis for $H^\ast (S,\C)$, define $\{\alpha_i^\vee\}_i$ to be the unique vectors satisfying $\int_S \alpha_j\alpha_i^\vee = \delta_{ij}$. Then $\sum_i\alpha_i\otimes\alpha_i^\vee$ is a Künneth decomposition.
One can choose the basis $B$ of $H^\ast (S,\C)$ in such a way that:
$\mathbbm{1}\in B$ and $\mathbbm{1}^\vee = \mathbf{p}\in B$; also $\sigma,\sigma^\vee=\overline{\sigma}\in B$ where $\sigma\in H^{2,0}(S)$;
$\gamma^\vee=\gamma\in B$ for all $\gamma = H^{1,1}(S)$. That means:
\begin{align}
    B=\{ \gamma,\gamma^\vee |\gamma\in H^0(S,\mathbb{C})\cup H^{2,0}(S,\mathbb{C}) \}\cup B^{1,1},
\end{align}
where $B^{1,1}=\{\gamma_1,\dots,\gamma_s\}$ is a basis of $H^{1,1}(S,\mathbb{C})$ such that $\int_S\gamma_i\gamma_j=\delta_{ij}$. This is possible since this symmetric paring is non-degenerate on $H^{1,1}(S,\mathbb{C})$ and can be diagonalized. The operator $T_{-k}$ now becomes:
\begin{align}
    T_{-k}= -\frac{1}{4}\sum_{v\in B} \sum_{\substack{i+j=k\\i>0\\j>0}}(-1)^{p_v}\operatorname{d}_j[v]\operatorname{d}_i[v^\vee] + \left(\frac{1}{4}\int_S \operatorname{td}_S\right)\sum_{\substack{i+j=k\\i>0\\j>0}}  \operatorname{d}_j[\boldsymbol{\operatorname{p}}]\operatorname{d}_i[\boldsymbol{\operatorname{p}}],
\end{align}
where $p_v$ is defined by $v\in H^{p_v,q_v}(S,\C)$.
Also define the operator $L_{-k}$ as:
\begin{align}
\label{knegative}
    L_{-k}=R_{-k}+T_{-k}.
\end{align}

In this subsection, I will prove the following proposition:
\begin{proposition}
\label{negvir}
    Let $S$ be a K3 surface and let the operators  $\{L_k:\mathbb{D}^S\to \mathbb{D}^S\}_{k\in\mathbb{Z}}$ be defined as \eqref{kpositive} when $k\geq -1$ and be defined as \eqref{knegative} when $k<-1$ then the following commutation relation is true:
    \begin{align}
        [L_{l},L_{-k}]=(-l-k)L_{l-k}+\frac{k^3-k}{12} \cdot e(S)\cdot\delta_{k,l},
    \end{align}
    where $e(S)$ is the topological Euler characteristic of $S$.
\end{proposition}

\begin{proof}
I also write $D\in \mathbb{D}^S$ as a monomial containing only descendants of the form $\ch{i}{v\in B}$. In order to clarify the action of $T_{-k}$, for any $D\in\mathbb{D}^S$, I can rewrite $D$ in the following form (since $S$ only has even cohomology, all descendants commute with each other):
\begin{align}
    D &= \prod_{v\in B}\prod_{\substack{k>i>0}} \left(\ch{i}{v}\right)^{n_{i,v}} \cdot D_0\cdot D_{\geq k}\\
    &= D_{<k}\cdot D_0\cdot D_{\geq k},
\end{align}
where $D_0,D_{<k},D_{\geq k}$ contains, respectively, only descendent of the form $\ch{0}{v\in B}$,$\ch{0<\bullet<k}{v\in B}$ and $\ch{\geq k}{v\in B}$. The reason why I write $D$ in this form is that $T_{-k}$ will only act non trivially on $D_{<k}$ part, and $R_{-k}$ will only act non trivially on $D_{\geq k}$ part.

Let $k>1$, $l\geq -1$ I calculate some commutators now:
\begin{align}
\label{com3}
\begin{split}
    [R_l,T_{-k}](D_{<k}D_0D_{\geq k})&= (R_lT_{-k}D_{<k})(D_0D_{\geq k})+(T_{-k}D_{<k})(R_lD_0D_{\geq k})\\&\quad\quad-(T_{-k}R_{l}D_{<k})(D_0D_{\geq k})-(T_{-k}D_{<k})(R_lD_0D_{\geq k})\\
    &=(R_lT_{-k}D_{<k})(D_0D_{\geq k}) - (T_{-k}R_{l}D_{<k})(D_0D_{\geq k})\\
    \end{split}
\end{align}
To calculate $R_lT_{-k}D_{<k}-T_{-k}R_{l}D_{<k}$, let me calculate the following first:
\begin{align}
\label{onecom}
    (R_ld_j[v]d_i[v^\vee]-\operatorname{d}_j[v]\operatorname{d}_i[v^\vee]R_l)D_{<k},
\end{align}
where $i,j$ are fixed integers satisfying $i+j=k,i>0,j>0$.
If $R_l$ acts on $\ch{i-l}{v}$ in $D_{<k}$ then its contribution to $-\operatorname{d}_j[v]\operatorname{d}_i[v^\vee]R_lD_{<k}$ is:
\begin{align}
\label{diffres1}
    -\frac{i!}{(i-l-1)!}\frac{1}{(i-1)!}\frac{1}{(j-1)!}n_{i-l,v}n_{j,v^\vee}(n_{i,v}+1) \frac{D_{<k}}{\ch{i-l}{v}\ch{j}{v^\vee}};
\end{align}
when $R_l$ acts on $\ch{j-l}{v^\vee}$, it gives:
\begin{align}
\label{diffres2}
    -\frac{j!}{(j-l-1)!}\frac{1}{(i-1)!}\frac{1}{(j-1)!}n_{j-l,v^\vee}(n_{j,v^\vee}+1)n_{i,v} \frac{D_{<k}}{\ch{j-l}{v^\vee}\ch{i}{v}};
\end{align}
when $R_l$ acts on $\ch{i}{v}$ or on $\ch{j}{v^\vee}$ it gives respectively:
\begin{align}
\label{samereslt1}
    &-\frac{(l+i)!}{(i-1)!}\frac{1}{(i-1)!}\frac{1}{(j-1)!}n_{i,v}n_{j,v^\vee}(n_{i,v}-1) \frac{D_{<k}\ch{i+l}{v}}{\ch{i}{v}\ch{j}{v^\vee}\ch{i}{v}}\quad\text{ and }\\
    \label{samereslt3}
    &-\frac{(l+j))!}{(j-1)!}\frac{1}{(i-1)!}\frac{1}{(j-1)!}n_{j,v^\vee}(n_{j,v^\vee}-1)n_{i,v} \frac{D_{<k}\ch{j+l}{v^\vee}}{\ch{j}{v^\vee}\ch{j}{v^\vee}\ch{i}{v}}
\end{align}
when $R_l$ acts on other descendants, say $\ch{p}{w}$, it gives:
\begin{align}
\label{samereslt2}
    -\frac{(l+p))!}{(p-1)!}\frac{1}{(i-1)!}\frac{1}{(j-1)!}n_{p,w}n_{j,v^\vee}n_{i,v} \frac{D_{<k}\ch{p+l}{w}}{\ch{p}{w}\ch{j}{v^\vee}\ch{i}{v}}.
\end{align}

Now, let us look at the first term of \eqref{onecom}:
\begin{align}
R_ld_j[v]d_i[v^\vee]D_{<k}=\frac{1}{(i-1)!}\frac{1}{(j-1)!}n_{j,v^\vee}n_{i,v}R_l\frac{D_{<k}}{\ch{j}{v^\vee}\ch{i}{v}},
\end{align}
It is easy to see that when $R_l$ does not act on $\ch{i-l}{v}$ and not on $\ch{j-l}{v^\vee}$, but on $\ch{i}{v},\ch{j}{v^\vee},\ch{p}{w}$ respectively, it gives $-\eqref{samereslt1}$, $-\eqref{samereslt3}$ and $-\eqref{samereslt2}$ as above. When $R_l$ acts on $\ch{i-l}{v}$ or on $\ch{j-l}{v^\vee}$ it gives:
\begin{align}
\label{diffres21}
    &\frac{i!}{(i-l-1)!}\frac{1}{(i-1)!}\frac{1}{(j-1)!}n_{i-l,v}n_{j,v^\vee}n_{i,v} \frac{D_{<k}}{\ch{i-l}{v}\ch{j}{v^\vee}}\quad \text{ and } \\
    \label{diffres22}
    &\frac{j!}{(j-l-1)!}\frac{1}{(i-1)!}\frac{1}{(j-1)!}n_{j-l,v^\vee}n_{j,v^\vee}n_{i,v} \frac{D_{<k}}{\ch{j-l}{v^\vee}\ch{i}{v}}.
\end{align}
Summing \eqref{diffres1}, \eqref{diffres2}, \eqref{diffres21} and \eqref{diffres22} gives:
\begin{align}
    \eqref{onecom}&= -\frac{i!}{(i-l-1)!}\frac{1}{(i-1)!}\frac{1}{(j-1)!}n_{i-l,v}n_{j,v^\vee} \frac{D_{<k}}{\ch{i-l}{v}\ch{j}{v^\vee}}\\
    &\quad\quad-\frac{j!}{(j-l-1)!}\frac{1}{(i-1)!}\frac{1}{(j-1)!}n_{j-l,v^\vee}n_{i,v} \frac{D_{<k}}{\ch{j-l}{v^\vee}\ch{i}{v}}\\
    &= -i\operatorname{d}_{i-l}[v^\vee]\operatorname{d}_{j}[v] D_{<k}-j\operatorname{d}_{i}[v^\vee]\operatorname{d}_{j-l}[v] D_{<k}
\end{align}

Summing terms of the form \eqref{onecom} with according coefficients over $i+j=k,i>0,j>0$ and over $B$ one gets:
\begin{align}
\label{complicatecal}
    [R_l,T_{-k}]= \begin{cases}
      (-l-k) T_{-(k-l)} , & \text{if}\ l-k<0 \\
      0, & \text{otherwise}.
    \end{cases}
\end{align}

Next, I calculate the following commutator:
\begin{align}
\label{com4}
\begin{split}
    [T_l,R_{-k}](D_{<k}D_0D_{\geq k})&= T_lD_{<k}D_0(R_{-k}D_{\geq k})-(R_{-k}T_l)D_{<k}D_0D_{\geq k}\\
    &\quad\quad\quad-T_lD_{<k}D_0(R_{-k}D_{\geq k})\\
    &=-(R_{-k}T_l)D_{<k}D_0D_{\geq k}\\
    \end{split}.
\end{align}
If $l<k$, $R_{-k}T_l$ vanishes. Suppose now $l\geq k$, then:
\begin{align}
    R_{-k}T_l&= R_{-k}\sum_{i+j=l} (-1)^{\operatorname{dim}(S)-p^L}i!j!\operatorname{ch}^{\operatorname{H}}_i\ch{j}{\operatorname{td}_S}\\
    &=\sum_{i+j=l}(-1)^{\operatorname{dim}(S)-p^L}\left( i!j!\frac{(i-k)!}{(i-1)!}\operatorname{ch}^{\operatorname{H}}_{i-k}\ch{j}{\operatorname{td}_S} + i!j!\frac{(j-k)!}{(j-1)!}\operatorname{ch}^{\operatorname{H}}_i\ch{j-k}{\operatorname{td}_S}\right)\\
    &=(l+k)T_{l-k},
\end{align}
where the convention $\ch{<0}{\bullet}=0$ is used.
Therefore:
\begin{align}
    [T_l,R_{-k}]=\begin{cases}
      (-l-k)T_{l-k} , & \text{if}\ l-k\geq0 \\
      0, & \text{otherwise}.
    \end{cases}
\end{align}

Next, I calculate $[T_{-k},T_l]D$. Generally, for two derivations operators $d_1,d_2$ one has: $[d_1d_2, T_l]D=(d_1d_2T_l)D+(d_1T_l)(d_2D)+(d_2T_l)(d_1D)$. Because the operator $T_{-k}$ is the sum of compositions of two derivatives, one can use this equality to calculate $[T_{-k},T_l]D$. As mentioned before, derivations $d_i[\bullet]$ with $ 0<i<k$ only act non trivially on the component $D_{<k}$, therefore we can perform the following calculations:

\begin{align}
\label{TT21}
\begin{split}
    &-\frac{1}{4}\sum_{v\in B} \sum_{\substack{i+j=k\\i>0\\j>0}}(-1)^{p_v} (\operatorname{d}_j[v]T_l)(\operatorname{d}_i[v^\vee]D_{<k}) \\
    &=  -\frac{1}{4}\sum_{v\in B} \sum_{\substack{i+j=k\\1\leq j \leq \operatorname{min}(k-1,l)}} \left(\frac{2}{(j-1)!}(-1)^{p_v}j!(l-j)!\ch{l-j}{v}\right)\left((-1)^{p_v}\frac{1}{(i-1)!}\cdot n_{i,v}\cdot \frac{ D_{<k}}{ \ch{i}{v}} \right) \\
    &\quad\quad-\frac{1}{4} \sum_{\substack{i+j=k\\1\leq j \leq \operatorname{min}(k-1,l)}}\left(\left(\int_S\operatorname{td}_S\right)\cdot \frac{2}{(j-1)!}j!(l-j)!\ch{l-j}{\boldsymbol{\operatorname{p}}} \right)\left( \frac{1}{(i-1)!}\cdot n_{i,\mathbbm{1}}\cdot \frac{D_{<k}}{\ch{i}{\mathbbm{1}}}\right)\\
    &= -\frac{1}{2}\sum_{v\in B} \sum_{\substack{i+j=k\\1\leq j \leq \operatorname{min}(k-1,l)}} \left(\frac{(i-(k-l))!}{(i-1)!}(i-k)\cdot n_{i,v}\cdot \frac{D_{<k}\ch{i-(k-l)}{v}}{\ch{i}{v}} \right) \\
    &\quad\quad- \frac{1}{2}\left(\int_S\operatorname{td}_S\right)\sum_{\substack{i+j=k\\1\leq j \leq \operatorname{min}(k-1,l)}} \left( \frac{(i-(k-l))!}{(i-1)!}jn_{i,\mathbbm{1}} \frac{D_{<k}\ch{l-j}{\boldsymbol{\operatorname{p}}}}{\ch{i}{\mathbbm{1}}}\right)
    \end{split}.
\end{align}
Similarly:
\begin{align}
\label{TT22}
\begin{split}
     &-\frac{1}{4}\sum_{v\in B} \sum_{\substack{i+j=k\\i>0\\j>0}}(-1)^{p_v}  (\operatorname{d}_j[v^\vee]T_l)(\operatorname{d}_i[v]D_{<k}) \\
     &= -\frac{1}{2}\sum_{v\in B} \sum_{\substack{i+j=k\\1\leq j \leq \operatorname{min}(k-1,l)}} \left(\frac{(i-(k-l))!}{(i-1)!}(i-k)\cdot n_{i,v^\vee}\cdot \frac{D_{<k}\ch{i-(k-l)}{v^\vee}}{\ch{i}{v^\vee}} \right) \\
    &\quad\quad- \frac{1}{2}\left(\int_S\operatorname{td}_S\right)\sum_{\substack{i+j=k\\1\leq j \leq \operatorname{min}(k-1,l)}} \left( \frac{(i-(k-l))!}{(i-1)!}jn_{i,\mathbbm{1}} \frac{D_{<k}\ch{l-j}{\boldsymbol{\operatorname{p}}}}{\ch{i}{\mathbbm{1}}}\right)\\
     \end{split}.
\end{align}
There is also the following terms:
\begin{align}
\label{TT23}
\begin{split}
    &\frac{1}{4}\left(\int_S\operatorname{td}_S\right)\sum_{\substack{i+j=k\\i>0\\j>0}}  (\operatorname{d}_j[\boldsymbol{\operatorname{p}}]T_l)(\operatorname{d}_i[\boldsymbol{\operatorname{p}}]D_{<k})\\
    &=\frac{1}{4}\left(\int_S\operatorname{td}_S\right)\sum_{\substack{i+j=k\\i>0\\j>0}} \left( \frac{2}{(j-1)!}j!(l-j)!\ch{l-j}{\boldsymbol{\operatorname{p}}}\right)\left( \frac{1}{(i-1)!}\cdot n_{i,\mathbbm{1}}\cdot \frac{D_{<k}}{\ch{i}{\mathbbm{1}}}\right)\\
    &=\frac{1}{2}\left(\int_S\operatorname{td}_S\right)\sum_{\substack{i+j=k\\1\leq j \leq \operatorname{min}(k-1,l)}} \left( \frac{(i-(k-l))!}{(i-1)!}jn_{i,\mathbbm{1}}\frac{D_{<k}\ch{l-j}{\boldsymbol{\operatorname{p}}}}{\ch{i}{\mathbbm{1}}} \right)
    \end{split}\\
    \begin{split}
         &\sum_{\substack{i+j=k\\i>0\\j>0}}  (\operatorname{d}_j[\boldsymbol{\operatorname{p}}]\operatorname{d}_i[\boldsymbol{\operatorname{p}}]T_k)(D_{<k})\\
    &=0
    \end{split}
\end{align}

Summing \eqref{TT21}, \eqref{TT22} and \eqref{TT23} one gets:
\begin{align}
\label{109}
    \begin{split}
        \eqref{TT21}+&\eqref{TT22}+2 \times\eqref{TT23}\\
        &=-\sum_{v\in B} \sum_{\substack{i+j=k\\1\leq j \leq \operatorname{min}(k-1,l)}} \left(\frac{(i-(k-l))!}{(i-1)!}(i-k)\cdot n_{i,v^\vee}\cdot \frac{D_{<k}\ch{i-(k-l)}{v^\vee}}{\ch{i}{v^\vee}} \right)\\
        &=-R_{l-k}[\bullet-k] D_{<k}
    \end{split},
\end{align}
where, for a function $f(\bullet)$ defined on integers, $R_j[f(\bullet)]$ is the derivation acting as $R_j[f(\bullet)]\ch{i}{\gamma}=f(i)R_j\ch{i}{\gamma}$.

There are also the following terms in $[T_{-k},T_l]$ which are non-zero only if $k=l$:
\begin{align}
\label{ddtd}
\begin{split}
    & -\frac{1}{4}\sum_{v\in B} \sum_{\substack{i+j=k\\i>0\\j>0}} (-1)^{p_v}(\operatorname{d}_j[v]\operatorname{d}_i[v^\vee]T_k)D_{<k}\\
    &=-\frac{1}{4}\sum_{v\in B} \sum_{\substack{i+j=k\\i>0\\j>0}} (-1)^{p_v}\left( 2 (-1)^{p_v}ij \right)D_{<k}\\
    &=-\frac{1}{2}\sum_{v\in B} \sum_{\substack{i+j=k\\i>0\\j>0}}ijD_{<k}\\
    &=-\frac{1}{2}\sum_{v\in B} \sum_{\substack{k>i>0}}(ik-i^2)D_{<k}\\
    &=-\sum_{v\in B}\frac{1}{2}\left(k\cdot \frac{k(k-1)}{2}-\frac{(k-1)k(2k-1)}{6} \right)D_{<k}\\
    &=-\frac{k^3-k}{12} \cdot e(S)D_{<k}, 
\end{split}
\end{align}
where $e(S)$ is the topological Euler characteristic and, for K3, $e(S) = |B|= \operatorname{dim}H^\ast(S,\C)=24$.

Next, I calculate the following commutator:
\begin{align}
\begin{split}
\label{com2}
     &[R_l,R_{-k}](D_{<k}D_0D_{\geq k})\\
    &= (R_lR_{-k}-R_{-k}R_l)D_{<k}D_0D_{\geq k}\\
    &=(R_lD_{<k}D_0)(R_{-k}D_{\geq k})+D_{<k}D_0( R_lR_{-k}D_{\geq k})-(R_{-k}R_lD_{<k}D_0)D_{\geq k}\\&\quad-(R_lD_{<k}D_0) (R_{-k}D_{\geq k})  -D_{<k}D_0 (R_{-k}R_lD_{\geq k})\\
    &=D_{<k}D_0( R_lR_{-k}D_{\geq k})  -D_{<k}D_0 (R_{-k}R_lD_{\geq k})-(R_{-k}R_lD_{<k}D_0)D_{\geq k}\\
    &=D_{<k}D_0((-k-l)R_{l-k}D_{\geq k})-(R_{l-k}[(\bullet+l)]D_{<k}D_0)D_{\geq k}.
\end{split}
\end{align}

Combining \eqref{109} with \eqref{com2}, one has (notice that $R_{k}\ch{0}{\bullet}=0,\forall k\in \mathbb{Z}$):
\begin{align}
\begin{split}
    -\eqref{109}+\eqref{com2} &= D_{<k}D_0((-k-l)R_{l-k}D_{\geq k})-(R_{l-k}[(\bullet+l)]D_{<k}D_0)D_{\geq k} \\
    &\quad\quad+ (R_{l-k}[\bullet-k] D_{<k} D_0)D_{\geq k}\\
    &=D_{<k}D_0((-k-l)R_{l-k}D_{\geq k}) + ((-k-l)R_{l-k} D_{<k} D_0)D_{\geq k}\\
    &=(-k-l)R_{l-k}(D_{<k} D_0D_{\geq k})
\end{split}
\end{align}

Summing the above commutators one has:
\begin{align}
    [L_{l},L_{-k}]=(-l-k)L_{l-k}+\frac{k^3-k}{12} \cdot e(S)\cdot\delta_{kl}
\end{align}

Let me calculate $[L_{-l},L_{-k}]$ next.
\begin{align}
    [R_{-l},R_{-k}]\ch{i}{\gamma}&=(R_{-l}R_{-k}-R_{-k}R_{-l})\ch{i}{\gamma}\\
    &=\frac{(i-k-l)!}{(i-k-1)!}\frac{(i-k)!}{(i-1)!}\ch{i-k-l}{\gamma}\\
    &\quad\quad\quad\quad-\frac{(i-k-l)!}{(i-l-1)!}\frac{(i-l)!}{(i-1)!}\ch{i-k-l}{\gamma}\\
    &=(l-k)R_{-k-l}\ch{i}{\gamma}, \quad\text{for $i\geq l+k$}.\\
     [R_{-l},R_{-k}]\ch{i}{\gamma}&=0, \quad\text{for $i< l+k$}.
\end{align}
Therefore $[R_{-l},R_{-k}]=R_{-k-l}$. Next, I calculate $[R_{-l},T_{-k}]D$. Firstly, consider the following term:
\begin{align}
\label{124}
    (R_{-l}\operatorname{d}_j[v^\vee]\operatorname{d}_{k-j}[v]-\operatorname{d}_j[v^\vee]\operatorname{d}_{k-j}[v]R_{-l})D.
\end{align}
Similar as the calculation for \eqref{complicatecal}, the contribution to the above term vanishes unless $R_{-l}$ acts on one of the following symbols: \[\ch{j}{v}, \ch{k-j}{v^\vee}, \ch{j+l}{v}, \ch{k-j+l}{v^\vee}.\]
By an easy calculation, one finds that when $R_{-l}$ acts on $\ch{j}{v}, \ch{k-j}{v^\vee}$, the contribution to \eqref{124} also vanishes. Let me calculate the case when $R_{-l}$ acts on $\ch{j+l}{v}$:
\begin{align*}
    R_{-l}&\operatorname{d}_j[v^\vee]\operatorname{d}_{k-j}[v]D \\&= \frac{1}{(j-1)!}\frac{1}{(k-j-1)!}n_{j,v}n_{k-j,v^\vee}n_{v,j+l}\frac{j!}{(j+l-1)!}\frac{D}{\ch{j+l}{v}\ch{k-j}{v^\vee}}\\
    \operatorname{d}_j[v^\vee]&\operatorname{d}_{k-j}[v]R_{-l}D\\&=\frac{1}{(j-1)!}\frac{1}{(k-j-1)!}(n_{j,v}+1)n_{k-j,v^\vee}n_{v,j+l}\frac{j!}{(j+l-1)!}\frac{D}{\ch{j+l}{v}\ch{k-j}{v^\vee}}\\
\end{align*}
Therefore:
\begin{align*}
    \eqref{124}=-j\operatorname{d}[v^\vee,j+l]\operatorname{d}[v,k-j] D.
\end{align*}
The contribution of $\ch{k-j+l}{v^\vee}$ can be calculated similarly. 
Summing over $i+j=k,i>0,j>0$ those two type of contributions, one has:
\begin{align}
\label{125}
    \sum_{\substack{i+j=k+l\\j\geq l+1\\i\leq k-1}}-(j-l)\operatorname{d}_j[v^\vee]\operatorname{d}_i[v] - \sum_{\substack{i+j=k+l\\i\geq l+1\\j\leq k-1}}(i-l)\operatorname{d}_j[v^\vee]\operatorname{d}_i[v].
\end{align}
For the commutator $[T_{-l},R_{-k}]D$, we have the following term:
\begin{align}
\label{126}
    \sum_{\substack{i+j=k+l\\j\geq k+1\\i\leq l-1}}(j-k)\operatorname{d}_j[v^\vee]\operatorname{d}_i[v] + \sum_{\substack{i+j=k+l\\i\geq k+1\\j\leq l-1}}(i-k)\operatorname{d}_j[v^\vee]\operatorname{d}_i[v]
\end{align}
Summing the first term of \eqref{125} and the second term of \eqref{126} and summing the second term of \eqref{125} and the first term of \eqref{126}:
\begin{align}
    \eqref{125}+\eqref{126} &= \sum_{\substack{i+j=k+l\\j\geq 1\\i\geq 1}}-(j-l)\operatorname{d}_j[v^\vee]\operatorname{d}_i[v]-\sum_{\substack{i+j=k+l\\j\geq 1\\i\geq 1}}(i-l)\operatorname{d}_j[v^\vee]\operatorname{d}_i[v]\\
    &=(l-k)\sum_{\substack{i+j=k+l\\j\geq 1\\i\geq 1}}\operatorname{d}_j[v^\vee]\operatorname{d}_i[v]
\end{align}
Summing over $v\in B$ with corresponding coefficients, one has:
\begin{align}
    [R_{-l},T_{-k}]+[T_{-l},R_{-k}]=(l-k)T_{-l-k}.
\end{align}
For the commutator $[T_{-l},T_{-k}]$, one can easily see that this vanishes. Therefore the following commutation relation is proven:
\begin{align}
    [L_{-l},L_{-k}]= (l-k)L_{-l-k}, \text{ for $l>1,k>1$}.
\end{align}
For $m,n\geq -1 $, it has been shown (lemma 2.11 \cite{vanbree2021virasoro}) that
\begin{align}
    [L_{m},L_{n}]= (-m+n)L_{m+n}, \text{ for $m,n\geq -1$}.
\end{align}
Therefore the Proposition \ref{negvir} is proven.
\end{proof}

\printbibliography

@article{bojko2022virasoro,
      title={Virasoro constraints for moduli of sheaves and vertex algebras}, 
      author={Arkadij Bojko and Woonam Lim and Miguel Moreira},
      year={2024},
      journal={Inventiones mathematicae},
}

@article{moreira2020virasoro, title={Virasoro constraints for stable pairs on toric threefolds}, volume={10}, journal={Forum of Mathematics, Pi}, author={Moreira, Miguel and Oblomkov, Alexei and Okounkov, Andrei and Pandharipande, Rahul}, year={2022}, pages={e20}}

@article{Oberdieck_2022,
  
	year = 2022,
	month = {oct},
  
	publisher = {{SIGMA} (Symmetry, Integrability and Geometry: Methods and Application)},
  
	author = {Georg Oberdieck},
  
	title = {Universality of Descendent Integrals over Moduli Spaces of Stable Sheaves on K3 Surfaces},
  
	journal = {Symmetry, Integrability and Geometry: Methods and Applications}
}

@article{markman2005monodromy,
      title={On the monodromy of moduli spaces of sheaves on K3 surfaces}, 
      author={Eyal Markman},
      journal={J. Algebraic Geom. 17, 29-99}, year={2008}}

@book{huybrechts_lehn_2010, place={Cambridge}, edition={2}, series={Cambridge Mathematical Library}, title={The Geometry of Moduli Spaces of Sheaves},  publisher={Cambridge University Press}, author={Huybrechts, Daniel and Lehn, Manfred}, year={2010}, collection={Cambridge Mathematical Library}}

@article{Eguchi_1997,
   title={Quantum cohomology and Virasoro algebra},
   volume={402},
   ISSN={0370-2693},
   
   number={1–2},
   journal={Physics Letters B},
   publisher={Elsevier BV},
   author={Eguchi, Tohru and Hori, Kentaro and Xiong, Chuan-Sheng},
   year={1997},
   month=jun, pages={71–80} }

@article{pandharipande2003questions,
      title={Three questions in Gromov-Witten theory}, 
      author={R. Pandharipande},
      year={2002},
      journal={Pro- ceedings of the ICM 2002, Vol. II, Higher Education Press},
}

@book{Hori:2003ic,
    author = "Hori, K. and Katz, S. and Klemm, A. and Pandharipande, R. and Thomas, R. and Vafa, C. and Vakil, R. and Zaslow, E.",
    title = "{Mirror symmetry}",
    publisher = "AMS",
    address = "Providence, USA",
    series = "Clay mathematics monographs",
    volume = "1",
    year = "2003"
}

@article{Pandharipande_2009,
   title={Curve counting via stable pairs in the derived category},
   volume={178},
   ISSN={1432-1297},
   
   number={2},
   journal={Inventiones mathematicae},
   publisher={Springer Science and Business Media LLC},
   author={Pandharipande, R. and Thomas, R. P.},
   year={2009},
   month=may, pages={407–447} }

@article{Moreira_2022,
   title={Virasoro conjecture for the stable pairs descendent theory of simply connected 3‐folds (with applications to the Hilbert scheme of points of a surface)},
   volume={106},
   ISSN={1469-7750},
   
   number={1},
   journal={Journal of the London Mathematical Society},
   publisher={Wiley},
   author={Moreira, Miguel},
   year={2022},
   month=mar, pages={154–191} }

@article{vanbree2021virasoro, title={Virasoro constraints for moduli spaces of sheaves on surfaces}, volume={11},  journal={Forum of Mathematics, Sigma}, author={van Bree, Dirk}, year={2023}, pages={e4}}
\end{document}